\documentclass[11pt]{amsart}
\usepackage[T1]{fontenc}
\usepackage{lmodern}
\usepackage[margin=1in]{geometry}
\usepackage{amsmath,amssymb,amsthm,mathtools}
\usepackage{microtype}
\usepackage{xcolor}
\usepackage[colorlinks=true,linkcolor=blue!50!black,citecolor=blue!50!black,urlcolor=blue!50!black]{hyperref}
\hypersetup{hypertexnames=false}
\usepackage{enumitem}
\usepackage{needspace}
\setlist{topsep=4pt,itemsep=3pt}
\newtheorem{theorem}{Theorem}[section]
\newtheorem{lemma}[theorem]{Lemma}
\newtheorem{proposition}[theorem]{Proposition}
\newtheorem{corollary}[theorem]{Corollary}
\theoremstyle{definition}

\theoremstyle{remark}
\newtheorem{remark}[theorem]{Remark}

\DeclareMathOperator{\supp}{supp}

\newcommand{\Proj}{\operatorname{Proj}}
\newcommand{\Lam}{\Lambda_{\Phi}}

\newcommand{\Pset}{\mathcal P}

\newcommand{\eps}{\varepsilon}
\newcommand{\nphi}{\mathfrak n_{\Phi}}
\newcommand{\mphi}{\mathfrak m_{\Phi}}
\newcommand{\pphi}{\mathfrak p_{\Phi}}
\newcommand{\Chi}{\mathbf 1}
\newcommand{\Dom}{\operatorname{Dom}}
\newcommand{\HH}{H_{\Phi}}
\newcommand{\Aphi}{A_{\Phi}}

\title{Quantum integrals\\ 
(With an Appendix on the solution to Haagerup's weight problem) }
\author{David P. Blecher}
\address{Department of Mathematics, University of Houston, Houston, TX
77204-3008, USA}
\email[David P. Blecher]{dpbleche@central.uh.edu}

\author{Louis Labuschagne}
\address{Internal Box 209, School of Comp., Stat. \& Math. Sci., NWU, Pvt. Bag X6001, 2520 Potchefstroom, South
Africa}
\email{louis.labuschagne@nwu.ac.za}

\date{October 4, 2026}
\begin{document}

\begin{abstract}  We solve two long outstanding problems here. 
We first solve in ZFC a problem that the first author and Weaver posed ten years ago: 
 a state on a von Neumann algebra is countably 
additive on orthogonal projections if and only if it is sequentially 
weak$^*$ continuous, equivalently, sequentially normal.  
Such states may be considered the quantum (countably additive) probability measures, 
or rather their extension to a noncommutative integral. 
The principal 
tools are known ideas 
from direct integral theory from the 1970s.
Indeed we prove  a disintegration theorem for states on products of 
$\sigma$-finite von Neumann algebras. A decomposition adapted to a given sequence reduces the 
continuity assertion in our main result to the classical dominated convergence 
theorem. The argument needs no 
separability hypothesis. 
We give several  applications.  For example we give a new variant of Gleason theorem, extending countably additive probability measures on projections to sequentially weak* continuous states when the algebra has no type $I_2$ direct summand. 

We also establish normality and sequential normality criteria for weights.
For example we prove GNS closability under sequential normality on abelian subalgebras. Under a certain semifiniteness condition, we obtain representations of sequentially normal weights by sequentially weak* continuous positive functionals. We  give several other applications of our main result to weights.  In forthcoming work we consider many applications to `quantum measure and integration theory', such as variants of Lebesgue's dominated convergence theorem for von Neumann algebras. 
 The analytic tools and product-quotient construction developed here are used  to prove in an Appendix by the first author that normality and sequential normality of arbitrary weights are determined by their restrictions to abelian von Neumann subalgebras, resolving Haagerup’s Problem 1.11.  (This may be moved to  a companion paper.) That is, normality of weights is commutatively determined.  \end{abstract}

\maketitle

\section{Introduction and main theorem}\label{sec:intro}

Let $M$ be a complex von Neumann algebra, with its canonical predual
$M_*$. 
A state $\varphi$ on $M$ is \emph{countably additive
on projections} if
\[
 \varphi\left(\sum_{n=1}^{\infty}p_n\right) 
   =\sum_{n=1}^{\infty}\varphi(p_n)
\]
for every sequence of mutually orthogonal projections in $M$. 
It is \emph{sequentially normal}, or \emph{$\sigma$-normal}, if
$\varphi(a_n)\uparrow\varphi(a)$ whenever $0\leq a_n\uparrow a$ in
$M$. Normality of $\varphi$ of course  requires the analogous assertion for bounded increasing
nets, not just sequences.

In the spirit of Gleason's celebrated and influential theorem, states on a von  
Neumann algebra which are countably additive on (mutually orthogonal) projections would seem to play the role of 
`quantum measures'.    
Our paper is a contribution to the theory of noncommutative measures which is closer in spirit to classical
measure theory in the sense of being based on countable additivity rather than normality.
We refer the reader to 
 \cite{Hamhalter} for an excellent treatise on many aspects of quantum measure theory.
Indeed a state on a commutative  von  
Neumann algebra $M = L^\infty(X,\mu)$ which is countably additive on projections is precisely integration with respect to a probability measure $\sigma \ll \mu$
(this follows e.g.\ from the proof of \cite[Lemma 3.2]{BGL} together with the fact that simple functions are norm dense in $M$).   It is easy to see that the latter 
state is sequentially normal.   In fact it is also sequentially weak* continuous.   (This follows from our main result here, but we give an independent proof of it in Theorem \ref{thm:mainc} below.)

For a state on a $\sigma$-finite von Neumann algebra, countable additivity on
projections already implies normality (simply because it is equivalent to complete  additivity on
projections). The question addressed in the first part of our paper  is
what survives without $\sigma$-finiteness. Our main result is that
countable additivity still controls every weak$^*$ convergent sequence,
even when the sequence is neither positive nor monotone.

\begin{theorem}[Sequential continuity]\label{thm:main}
For a state $\varphi$ on a von Neumann algebra $M$, the following
are equivalent:
\begin{enumerate}[label=(\roman*)]
\item [{\rm (i)}] $\varphi$ is countably additive on  projections.
\item [{\rm (ii)}] $\varphi$ is sequentially normal.
\item [{\rm (iii)}] $\varphi$ is sequentially weak$^*$ continuous: if $a_n\to a$
weak$^*$ in $M$, then $\varphi(a_n)\to\varphi(a)$.
\end{enumerate}
In any faithful normal concrete representation of $M$, these conditions
are also equivalent to sequential continuity for the weak and strong operator
topologies.
\end{theorem}

Thus sequential weak$^*$ continuity of states  is commutatively determined.   Indeed a positive linear functional $\varphi$ on $M$ is weak* sequentially continuous or  sequentially  
normal on $M$ if and only if  $\varphi|_A$ has these properties for every unital abelian von Neumann
      subalgebra $A\subseteq M$. 
We now have several quite different proofs of this result, however the one below is the shortest.
(We do have shorter, or quite interesting, proofs of special cases of it, or of complementary results such as Corollary
\ref{cor} below.)  
As a first application we give a new variant of Gleason's theorem, describing the countably additive
projection measures for any  von Neumann algebra with
no type $I_2$ direct summand. 

We then study weights, obtaining direct normality and sequential
normality criteria under additional hypotheses, including certain variants of semifiniteness 
of the weight.  We also develop
normal-minorant constructions and GNS closability methods. These
arguments are proved independently here and provide ingredients for
the solution in the Appendix of Haagerup's problem 1.11 from \cite{Haag} and its sequential variant.

The most nontrivial implication in Theorem \ref{thm:main} 
is (i)$\Rightarrow$(iii). Its proof uses
two reductions. First, we disintegrate a countably additive state on a
product of $\sigma$-finite von Neumann algebras into normal states on
the factors. The measure is determined by the state on the central
indicator projections, and it is defined on the entire power set of the
index set. Second, we replace $M$ by the von Neumann algebra generated by the
sequence. This algebra is a product of at most a continuum of separably
acting, hence $\sigma$-finite, von Neumann algebras. Normality
in each factor and the scalar dominated convergence theorem then prove the desired
continuity.
The decomposition may depend on the sequence, but this
 is harmless for us since we have fixed the sequence.

In \cite{BlecherWeaver} we study countably additive states
and their relation to measurable cardinals.  Theorem~4.6 of   \cite{BlecherWeaver}  proves
sequential weak$^*$ continuity for countably additive pure states on
arbitrary von Neumann algebras. However Theorem~\ref{thm:main} does not assume
purity. For background on quantum measures and Gleason extension
theorems, see Hamhalter's monograph
\cite[pp.~59--93 and 121--175]{Hamhalter}. The 
product construction is also related to the nonseparable direct-integral 
theory of Wils and Vesterstr{\o}m--Wils \cite{Wils,VW} and to 
Halpern's integral representations and decompositions 
\cite{Halpern1966,Halpern1972,Halpern1978}. Here we give the quotient 
and predual arguments explicitly, allowing the $\sigma$-finite 
fibers to vary. A crucial feature is that the original state is assumed 
to only be countably additive on projections, not normal. 

In Section~\ref{class} we give a short proof  of the commutative case of the main theorem, which in this setting becomes a seemingly new result in classical integration theory.  
We refer the reader to \cite{BGL}
for a description of several aspects of classical measure and integration theory from a von Neumann algebra perspective. 
Section~\ref{sec:prelim} contains the operator-algebraic preliminaries.
Section~\ref{sec:product} proves product disintegration, including
almost-everywhere uniqueness of the field of normal states.
Section~\ref{sec:continuity} proves Theorem~\ref{thm:main}, and combines this theorem with the 
Bunce--Wright extension theorem to obtain a sequentially continuous 
Gleason extension. 

Section~6 develops the weight theory and the analytic tools which we will also use in the sequel 
\cite{Bd}. That paper proves, for every weight $\Phi$ on an
arbitrary von Neumann algebra,
\[
 \Phi\text{ is normal}\ \Longleftrightarrow\ (H),
 \qquad
 \Phi\text{ is }\sigma\text{-normal}
       \ \Longleftrightarrow\ (H_\sigma),
\]
where $(H)$ and $(H_\sigma)$ mean normality and sequential normality,
respectively, on every unital abelian von Neumann subalgebra. (The sequential version may be moved to the present paper.) 
In
addition to the GNS and minorant arguments of Section~6, the sequential
proof uses the quotient construction of Lemma~\ref{lem:quotient} and the countably
generated reduction of Lemma~\ref{lem:countable-reduction}. 
  (We remark that Haagerup does characterize normality of weights there in terms of  `complete additivity', but this additivity is on positive elements, not projections). 
  Recently I.\ Choi and coauthors \cite{IC1,IC2} have solved two other problems from  \cite{Haag}, Problems 2.7 and 1.10 there, but we have not yet been able to link our work to theirs.     Unfortunately it seems that there are significant obstacles  to any attempt to extend the disintegration/integral representation theorem which was key to our solution to Theorem \ref{thm:main} to weights.   Nonetheless we shall still give several applications of Theorem \ref{thm:main} to weights. 
  In Section 7 we discuss some further special cases. 
  Section 8 is a closing summary.

In forthcoming work we consider many applications, such as variants of Lebesgue's dominated convergence theorem (LDCT) for von Neumann algebras \cite{BNcm}.    
Note that the classical dominated convergence theorem LDCT is a statement about sequential weak* convergence, namely  it maybe rephrased as 
 $\psi_g(f_n/g) \to \psi_g(f/g)$ where $\psi_g$ is the weak* continuous functional on $L^\infty(\mu)$ associated to the majorant $g \in L^1(\mu)$ of the sequence $(f_n)$.  
Indeed LDCT is equivalent to the assertion 
that every norm-bounded sequence $(f_n)$ in $M=L^\infty(X,\mu)$ which 
converges a.e.\ also converges weak*.  Or equivalently that $\psi(f_n) \to \psi(f)$ for every state $\psi$ of the kind characterized in Theorem \ref{thm:main}. 
In  \cite{BNcm} we also discuss the relations between sequential weak* convergence and other notions of sequential convergence such as strong or (local) almost uniform sequential  convergence. 

\begin{corollary}[Monotone convergence]\label{cor} 
For a positive linear functional $\phi$ on a von Neumann algebra 
$M$, countable additivity on projections is equivalent to 
$\sigma$-normality. 
\end{corollary}

\begin{proof} 
We stated this theorem in \cite{B18} without proof.   It follows immediately from Theorem~\ref{thm:main}.
We sketch an alternative proof that we will present in detail elsewhere.  
The abelian case is easy.   Here is one proof (another
uses 
the classical LMCT): if $f_n \downarrow 0$ in Ball$(M)_+$ and $t > 0$ 
let $e_n$ be the characteristic function $[f_n \geq t ]$.  Then $e_n \downarrow e \in {\rm Proj}(M)$.
Since $t e  \leq t e_n  \leq f_n \downarrow 0$ we have $e = 0$.  So $\phi(e_n) \downarrow 0$.  Since $f_n \leq t + e_n$
we have $\phi(f_n) \leq t \phi(1) + \phi(e_n)$, so that $0 \leq \limsup_n \phi(f_n) \leq t \phi(1)$.  Since $t>0$ was arbitrary we have
$\phi(f_n) \to 0$ as desired. 

It suffices to prove the result for finite and for properly infinite von Neumann algebras.
The former can be easily deduced from the abelian case above using the center valued trace.   The latter 
follows by the increasing sequence of projections case, using a special factorization property of increasing sequences in   properly infinite algebras.  \end{proof}  

\begin{corollary} \label{cor2} Weak* sequential continuity and $\sigma$-normality of states on a  von Neumann algebra $M$ are commutatively determined.   Indeed a positive linear functional $\varphi$ on $M$ is weak* sequentially continuous or $\sigma$-normal on $M$ if and only if  $\varphi|_A$ has these properties for every unital abelian von Neumann
      subalgebra $A\subseteq M$.
 \end{corollary}

\begin{proof}  Being countably additive on projections is commutatively determined.   Indeed,
if $(p_n)$ are mutually orthogonal projections in $M$ then
the unital von Neumann subalgebra $A=W^*(1,p_1,p_2,\ldots)$ is abelian. \end{proof}

\section{A new result in classical integration theory?}\label{class}

In this section we give a short proof  of the commutative case of the main theorem, which in this setting, and using a fact from the second paragraph of our paper,
 becomes a seemingly new result in classical integration theory (it  was mentioned in \cite{B18}, but we are not
aware of this result proved rigorously in the literature in this generality).

\begin{theorem}\label{thm:mainc}
Let $(X,\Sigma,\mu)$ be a semifinite measure space, and let
$\nu$ be a finite positive measure on $(X,\Sigma)$ with
$\nu\ll\mu$ (we remark that the latter  need not imply 
that there is a $g \in L^1$ with $d \nu = g d \mu$). Suppose that $(f_n)\subseteq L^\infty(\mu)$ and
\[
 \int_X f_nh\,d\mu\longrightarrow0
 \qquad\text{for every }h\in L^1(\mu).
\]
Then
\[
 \int_X f_n\,d\nu\longrightarrow0.
\]
The latter integrals are well defined on $L^\infty(\mu)$
because $\nu\ll\mu$.
\end{theorem}

\begin{proof} It suffices to treat real-valued functions, since taking real
and imaginary parts preserves the assumed convergence.
Let $(f_n)\subseteq L^\infty(\mu)$ be a real-valued sequence
satisfying the stated conditions.

Semifiniteness implies that $L^1(\mu)$ norms $L^\infty(\mu)$:
\[
 \|f\|_\infty
 =\sup_{\|h\|_1\le1}\left|\int_X fh\,d\mu\right|.
\]
The bounded linear functionals
\[
 T_n:L^1(\mu)\longrightarrow\mathbb C,
 \qquad T_n(h)=\int_X f_nh\,d\mu,
\]
converge pointwise to zero. The uniform boundedness principle,
together with the norming identity 
above,  therefore gives
\[
 C:=\sup_n\|f_n\|_\infty
   =\sup_n\|T_n\|<\infty.
\]

Choose real-valued measurable representatives satisfying
$|f_n|\le C$ everywhere, and keep these representatives fixed.
If the conclusion fails, then, after passing to a subsequence
and replacing all its terms by their negatives if necessary,
we may assume that, for some $\varepsilon>0$,
\[
 \int_X f_n\,d\nu\ge\varepsilon
 \qquad(n\ge1).
\]
Since $\nu$ is finite, $(f_n)$ is bounded in $L^2(\nu)$.
Weak sequential compactness in this Hilbert space yields 
a further subsequence, relabeled $(f_n)$, such that
$f_n\to g$ weakly in $L^2(\nu)$. 
The limit is real-valued. Integrating against $1\in L^2(\nu)$
gives
\[
 \int_X g\,d\nu
 =\lim_n\int_X f_n\,d\nu\ge\varepsilon.
\]

By Mazur's lemma, for each $k$ we may choose a finite convex
combination
\[
 v_k\in\operatorname{co}\{f_n:n\ge k\},
 \qquad
 \|v_k-g\|_{L^2(\nu)}\le 2^{-k}.
\]
We realize each $v_k$ as a measurable function by taking the
corresponding convex combination of the fixed representatives
of the $f_n$, and choose a measurable representative of $g$.
Since convergence in $L^2(\nu)$ implies an a.e.\ 
convergent subsequence, we may pass to a subsequence and
relabel so that $v_k\to g$ $\nu$-a.e..
Both properties in the preceding displayed equation remain valid:
the $k$th retained term has original index at least $k$.
The tail condition ensures that $v_k\to0$ in
$\sigma(L^\infty(\mu),L^1(\mu))$. Indeed, for every
$h\in L^1(\mu)$,
\[
 \left|\int_X v_kh\,d\mu\right|
 \le \sup_{n\ge k}\left|\int_X f_nh\,d\mu\right|
 \longrightarrow0.
\]
(Alternatively, the Hilbert space Banach--Saks theorem gives
arithmetic means of a subsequence converging to $g$ in
$L^2(\nu)$.  These means also converge to zero in
$\sigma(L^\infty(\mu),L^1(\mu))$ by scalar Ces\`aro convergence, and one then passes to an a.e.\ 
convergent subsequence as above.)
 
Since $\int_X g\,d\nu>0$, choose $\delta>0$ such that
$\nu(\{g>2\delta\})>0$. Put
\[
 E_m=\bigcap_{k\ge m}\{v_k>\delta\}.
\]
At $\nu$-almost every point where $g>2\delta$, convergence
$v_k\to g$ implies $v_k>\delta$ eventually. Hence
\[
 \nu\!\left(\bigcup_{m=1}^{\infty}E_m\right)>0,
\]
so $\nu(E_m)>0$ for some $m$. Absolute continuity gives
$\mu(E_m)>0$, and semifiniteness yields a measurable
$F\subseteq E_m$ with $0<\mu(F)<\infty$.

For every $k\ge m$ we have $v_k>\delta$ on $F$, and therefore
\[
 \int_F v_k\,d\mu\ge\delta\mu(F)>0.
\]
On the other hand, $1_F\in L^1(\mu)$, so the convergence
already proved for $(v_k)$ gives
$\int_Fv_k\,d\mu\to0$. This contradiction proves the result.
\end{proof}

{\bf Remark.} The convergence hypothesis is convergence to zero in
$\sigma(L^\infty(\mu),L^1(\mu))$. When $\mu$ is localizable,
the canonical identification
$L^\infty(\mu)=(L^1(\mu))^*$ makes this the weak$^*$ topology.
The theorem itself requires only semifiniteness.

\bigskip

The last  result will have many applications.  We give one, to Fremlin's notion of measure convergence \cite[Volume 3]{Frem}
(see particularly Chapter 32, and 363, and 367).  We may discuss other applications elsewhere.   

\begin{corollary}[Sequential continuity on the measure algebra]
\label{cor:measure-algebra-sequential-continuity}
Let $(X,\Sigma,\mu)$ be a semifinite measure space, and let $\nu$
be a finite measure on $\Sigma$ such that $\nu\ll\mu$.
Let $\mathfrak A=\Sigma/\mathcal N_\mu$ be the measure algebra,
equipped with its measure-algebra topology, and define
$\bar\nu([E])=\nu(E)$.
Then
\[
 a_n\longrightarrow a
 \quad\Longrightarrow\quad
 \bar\nu(a_n\triangle a)\longrightarrow0.
\]
In particular, $\bar\nu$ is sequentially continuous.
\end{corollary}

\begin{proof}
Choose $E_n\in\Sigma$ representing $a_n\triangle a$.
By the definition of the measure-algebra topology,
\[
 \int_F\mathbf1_{E_n}\,d\mu
 =\mu(F\cap E_n)\longrightarrow0
 \qquad\text{whenever }F\in\Sigma,\ \mu(F)<\infty.
\]
Finite linear combinations of indicators of finite-measure sets
are dense in $L^1(\mu)$. Since
$\|\mathbf1_{E_n}\|_\infty\le1$, approximation therefore gives
\[
 \int_X g\mathbf1_{E_n}\,d\mu\longrightarrow0
 \qquad g\in L^1(\mu).
\]
Thus $\mathbf1_{E_n}\to0$ in
$\sigma(L^\infty(\mu),L^1(\mu))$. 
Then apply the preceding commutative sequential-continuity theorem
to the sequentially normal positive functional
$ \Phi(f)=\int_X f\,d\nu$, for $f\in L^\infty(\mu)$.
This functional is well defined because $\nu\ll\mu$.
We obtain
\[
 \bar\nu(a_n\triangle a)
 =\Phi(\mathbf1_{E_n})\longrightarrow0.
\]
Finally,
$|\bar\nu(a_n)-\bar\nu(a)|
 \le\bar\nu(a_n\triangle a)\longrightarrow0$ 
 as desired.
\end{proof}

\section{Operator-algebraic preliminaries}\label{sec:prelim}
We use the following standard facts from von Neumann algebra theory
\cite{Sakai,Pedersen}. A positive functional on a von Neumann algebra is
normal if and only if it is completely additive on orthogonal
projections. Normal functionals separate points, have polar and Jordan
decompositions, and the support of a normal positive functional $\omega$  is the
least projection on which $\omega(p) = 1$. We also use the
abstract monotone characterization: a monotone closed or monotone  complete $C^*$-algebra
with a separating family of normal positive functionals is a von Neumann
algebra \cite[Theorem III.3.16]{Takesaki}. Here normality on a monotone complete algebra means preservation
of suprema of bounded increasing nets of positive elements.  For
positive functionals on other ordered $C^*$-algebras, we use
\emph{sequentially normal} in the same order-theoretic sense as in
Section~\ref{sec:intro}.
A von Neumann algebra is \emph{$\sigma$-finite} if it admits a faithful
normal state.

\begin{lemma}\label{lem:sigma}
Every countably additive state on a $\sigma$-finite von Neumann algebra
is normal.
\end{lemma}
\begin{proof} As we said in the Introduction this is immediate 
simply because it is equivalent to complete additivity on projections.  \end{proof}

The following is probably also well known.

\begin{lemma} \label{lem:selection}
Let $A$ be a unital monotone sequentially complete $C^*$-algebra with a
faithful positive functional $\Theta$ preserving suprema of bounded
increasing positive sequences. Then $A$ is monotone complete, and each
bounded upward-directed subset of $A_{\rm sa}$ has an increasing
sequence with the same supremum. Every sequentially normal positive
functional on $A$ is therefore normal.
\end{lemma}
\begin{proof}
For a general bounded upward-directed subset $D$ of $A_{\rm sa}$,
fix $d_0\in D$, consider the cofinal subset
\[
   D_0=\{d\in D:d\geq d_0\}.
\]
By adding a fixed element $x$ say, that is by replacing $D_0$ by $D_0 + x$, it suffices to consider
$D\subseteq A_+$. 
Let
$t=\sup_{d\in D}\Theta(d)$ and choose an increasing sequence
$d_n\in D$ with $\Theta(d_n)\to t$. Put $d=\sup_nd_n$.
For any $b\in D$, directedness gives an increasing sequence $c_n\in D$
with $c_n\geq d_n,b$. Let $c=\sup_nc_n$. Then $c\geq d,b$, while
\[
 \Theta(c)=\lim_n\Theta(c_n)\leq t=\Theta(d).
\]
Faithfulness forces $c=d$, so $b\leq d$. Thus $d=\sup D$.  So $A$ is monotone complete.

If $F$ is any sequentially normal positive functional, then
\[
 F(d) =   F(\sup D)=\sup_nF(d_n) \leq \sup_{b\in D}F(b) \leq F(d),
\]
which proves normality.
\end{proof}

The already monotone complete version of this
selection principle is discussed in \cite{Thill}.

\section{Disintegration over an arbitrary power-set measure}\label{sec:product}
The next few sections use the following notation.
 Let $(X,\Pset(X),\mu)$ be a probability space and let $N_x$ be nonzero
$\sigma$-finite von Neumann algebras. Their product, the usual 
von Neumann algebra direct sum $N=\oplus_{x \in X} \, N_x$ 
consists of all uniformly bounded fields, with coordinatewise operations.  Put
\[
 J_\mu=\{a\in N:a_x=0\text{ for }\mu\text{-almost every }x\},
 \qquad A_\mu=N/J_\mu.
\]
The ideal $J_\mu$ is norm closed.  Indeed a norm-convergent sequence of its
members vanishes off the union of their countably many exceptional null
sets, and pointwise norm convergence outside that union forces its limit
to vanish there as well. It is plainly a self-adjoint two-sided ideal.
Thus $A_\mu$ is a unital $C^*$-algebra.

All integrals below are as encountered in a graduate real variables course, however with the $\sigma$-algebra being the 
power set. Since every subset
of $X$ is measurable, every scalar-valued function of $x$ that occurs
is measurable. 

\paragraph{Relation to  direct-integral theory.}
The construction is closely related to the known nonseparable direct-integral
theory.  See e.g.\ Wils \cite{Wils} and Vesterstr{\o}m--Wils
\cite[Theorems~4.1 and~4.3]{VW}. For a fixed $W^*$-algebra, the latter
results realize the essentially bounded weakly measurable operator
fields as a von Neumann algebra and identify the resulting algebra
with a von Neumann tensor product. We give a direct proof because the
fibers $N_x$ here may vary and the full power-set $\sigma$-algebra
admits every bounded field without an additional measurable-field
hypothesis. The proof also gives the sequential normality of the
quotient map and of the integral functionals.

Halpern's papers \cite{Halpern1966,Halpern1972,Halpern1978} provide
further context for integral representations of normal functionals and
decompositions of von Neumann algebras, including nonseparable cases.
In particular, \cite{Halpern1978} treats algebras not countably
generated over their centers. In contrast, in Theorem~\ref{thm:product} below, the
state is assumed to be merely countably additive on projections.   
Thus the quotient and measurable-field constructions draw heavily on the
generalized direct integral theory from the 1970s.  The novel point  here is that the
initial state is assumed only to be countably additive on projections,
rather than normal. 

\subsection{The quotient von Neumann algebra}

\begin{lemma}\label{lem:quotient}
The algebra $A_\mu$ is a $\sigma$-finite von Neumann algebra.
The quotient map preserves suprema of bounded increasing sequences of
self-adjoint elements.
\end{lemma}
\begin{proof}
The quotient norm is
\[
 \|[a]\|=\mathop{\rm ess\,sup}_{x\in X}\|a_x\|.
\]
For the upper bound here, if $K$ is larger than the essential supremum,
replace $a_x$ by zero on the null set $\{x:\|a_x\|>K\}$; the resulting
field represents $[a]$ and has norm at most $K$. For the lower bound,
if $j\in J_\mu$, then $j_x=0$ outside a null set, so
$\|a+j\|\geq\mathop{\rm ess\,sup}_x\|a_x\|$. Taking the infimum over
$j$ proves the formula. Similarly, for a self-adjoint field $a$, the
class $[a]$ is positive exactly when $a_x\geq0$ almost everywhere. If
$a_x\geq0$ almost everywhere, set $c_x=a_x^{1/2}$ there and $c_x=0$
on the exceptional null set; then $[a]=[c]^*[c]\geq0$. Conversely, if
$[a]\geq0$, its negative part in the quotient is zero. The functional
calculus commutes with the quotient map  \cite[Proposition II.2.3.2(iv)]{Black}, and so $[a_-]=[a]_-=0$. Thus
$a_-=0$ almost everywhere,  with $a_x\geq0$ almost everywhere.

Now let $0\leq u_n\leq C1$ be increasing in $A_\mu$. Choose
self-adjoint representatives $c^{(n)}$. Positivity of $u_n$ and
$C1-u_n$ implies that, outside a null set depending on $n$,
$0\leq c_x^{(n)}\leq C1$.  Positivity of $u_{n+1}-u_n$ implies that
$c_x^{(n)}\leq c_x^{(n+1)}$ outside another null set. The union $Z$ of
all these countably many null sets is null. Define
\[
 a_x^{(n)}=\begin{cases}c_x^{(n)},&x\notin N,\\0,&x\in Z.
 \end{cases}
\]
Then $a^{(n)}$ still represents $u_n$, and
$0\leq a_x^{(n)}\leq a_x^{(n+1)}\leq C1$ for every $x$. Define
$a_x=\sup_na_x^{(n)}$ in the monotone-complete algebra $N_x$. The
field $a=(a_x)$ is uniformly bounded by $C$, so it belongs to $N$, and
$[a]$ is coordinatewise an upper bound for all $u_n$. If $[b]$ is any
other upper bound, first take
a self-adjoint representative $b$. For every $n$, positivity of
$[b-a^{(n)}]$ gives $b_x\geq a_x^{(n)}$ off a null set. Outside the
union of these null sets, $b_x\geq\sup_na_x^{(n)}=a_x$. Thus
$[b]\geq[a]$, proving that $[a]=\sup_nu_n$. This also proves directly
that the quotient map preserves the indicated suprema.

Choose faithful normal states $\theta_x$ on $N_x$. The state
\[
 \Theta([a])=\int_X\theta_x(a_x)\,d\mu(x)
\]
is well defined because changing $a$ on a null set does not change the
integral. The scalar function $x\mapsto\theta_x(a_x)$ is measurable
 since every subset of $X$ is
measurable, and it is bounded by $\|a\|$. If $[a]\geq0$, choose a representative which is positive at every
coordinate by setting it equal to zero on its exceptional null set.
If $\Theta([a])=0$, then $\theta_x(a_x)=0$ almost everywhere;
faithfulness of every $\theta_x$ gives $a_x=0$ there, so $[a]=0$.
Thus $\Theta$ is faithful. The coordinate description of suprema above
also gives a complete proof of sequential normality.  Namely, suppose
$0\leq u_n\uparrow u$ in $A_\mu$.  The construction above gives
representatives
\[
 0\leq a_x^{(1)}\leq a_x^{(2)}\leq\cdots,\qquad
 a_x^{(n)}\uparrow a_x
\]
for $u_n$ and $u$, respectively, outside one common null set; set all
these representatives equal to zero on that null set.  Since each
$\theta_x$ is normal,
$\theta_x(a_x^{(n)})\uparrow\theta_x(a_x)$ pointwise.  The scalar
monotone convergence theorem now gives
\[
 \Theta(u)
 =\int_X\theta_x(a_x)\,d\mu(x)
 =\lim_{n\to\infty}\int_X\theta_x(a_x^{(n)})\,d\mu(x)
 =\lim_{n\to\infty}\Theta(u_n).
\]
This is exactly sequential normality of $\Theta$.
Lemma~\ref{lem:selection} shows that
$A_\mu$ is monotone complete and $\Theta$ is normal.

For $b=[(b_x)]\in A_\mu$, the positive functional
$\Theta_b(u)=\Theta(b^*ub)$ is sequentially normal: if $u_n\uparrow u$,
the representatives just constructed satisfy
$b_x^*a_x^{(n)}b_x\uparrow b_x^*a_xb_x$ at every coordinate, and
monotone convergence applies. It is therefore normal by
Lemma~\ref{lem:selection}. These functionals separate the self-adjoint
part. Indeed, let $u=u^*\ne0$ and choose
$\lambda\in\sigma(u)\setminus\{0\}$.  Choose $\eps>0$ and a real
continuous function $g$ on $\sigma(u)$ such that $g(\lambda)\ne0$ and
whose support lies in $(\eps,\infty)$ if $\lambda>0$, or in
$(-\infty,-\eps)$ if $\lambda<0$.  For $b=g(u)$, functional calculus
gives
\[
 b^*ub=u\,g(u)^2.
\]
This element is nonzero because the continuous function
$t\mapsto t g(t)^2$ does not vanish at $\lambda$.  Moreover,
$b^*ub\geq\eps b^*b$ in the first case and
$b^*ub\leq-\eps b^*b$ in the second.  Thus it has a fixed sign, and
faithfulness shows that $\Theta(b^*ub)$ has the corresponding strict
sign.  
The aforementioned abstract
monotone characterization \cite[Theorem III.3.16]{Takesaki} now makes $A_\mu$ a von Neumann algebra.
Since it has a faithful normal state it is $\sigma$-finite.
\end{proof}

\subsection{Identification of the predual}

\begin{lemma}[The predual as integrable fields]\label{lem:predual}
With notation above, let $\mathcal E$ be the space of fields $f_x\in(N_x)_*$ with
$\int_X\|f_x\|\,d\mu(x)<\infty$, modulo almost-everywhere equality.
Then
\begin{equation}\label{eq:predual}
 \mathcal E\longrightarrow(A_\mu)_*,\qquad
 f\longmapsto F_f,\qquad
 F_f([a])=\int_Xf_x(a_x)\,d\mu(x)
\end{equation}
is a surjective linear isometry. Moreover, $F_f\geq0$ if and only if
$f_x\geq0$ almost everywhere.
\end{lemma}
\begin{proof}
Equip $\mathcal E$ with the norm $\|f\|_1=\int\|f_x\|\,d\mu$. The space $\mathcal E$ is complete. Indeed, a Cauchy
sequence has a subsequence $f^{(n)}$
with $\sum_n\|f^{(n+1)}-f^{(n)}\|_1<\infty$. By monotone convergence,
\[
 \sum_n\|f_x^{(n+1)}-f_x^{(n)}\|<\infty
\]
almost everywhere. Completeness of each $(N_x)_*$ gives a pointwise
limit $f_x$, and the integrals of the tails of the series give convergence
in $\mathcal E$. The original Cauchy sequence has the same limit.

The displayed formula is well defined.  Changing the field on a null
set does not alter the integral, and
$|f_x(a_x)|\leq\|f_x\|\|a_x\|$ gives
$\|F_f\|\leq\|f\|_1$. For each $x$, weak$^*$ compactness of the unit ball
of $N_x$ and weak$^*$ continuity of the normal functional $f_x$ give an element $b_x$ in the unit
 ball with $|f_x(b_x)|=\|f_x\|$. By multiplying $b_x$
by a suitable scalar of modulus one, we may obtain an  $a_x$ so that
$f_x(a_x)=\|f_x\|$. The field $(a_x)$ has norm at most one and hence
belongs to $N$. 
Therefore
$F_f([a])=\int\|f_x\|\,d\mu$ and $\|F_f\|=\|f\|_1$.

Suppose first that $f_x\geq0$ almost everywhere.  If $[a]\geq0$, choose
a representative with $a_x\geq0$ at every coordinate.  Then
$F_f([a])=\int_Xf_x(a_x)\,d\mu(x)\geq0$.  Thus $F_f$ is positive.
For sequential normality, let $0\leq u_n\uparrow u$
in $A_\mu$.  The coordinatewise construction in
Lemma~\ref{lem:quotient} gives representatives
$0\leq a_x^{(n)}\uparrow a_x$ of $u_n$ and $u$, after altering all
the fields on a common null set.  Since $f_x$ is a normal positive
functional on $N_x$,
\[
 f_x(a_x^{(n)})\uparrow f_x(a_x)
\]
for almost every $x$.  Scalar monotone convergence therefore gives
\[
 F_f(u)
   =\int_X f_x(a_x)\,d\mu(x)
   =\lim_{n\to\infty}\int_X f_x(a_x^{(n)})\,d\mu(x)
   =\lim_{n\to\infty}F_f(u_n).
\]
Thus $F_f$ preserves suprema of bounded increasing positive sequences.
By Lemma~\ref{lem:quotient} $A_\mu$ is a
$\sigma$-finite von Neumann algebra.  The proof of that lemma shows that  its state $\Theta$ is
faithful and normal. In particular, $A_\mu$ is monotone sequentially
complete and $\Theta$ preserves bounded increasing sequential suprema.
Hence by Lemma~\ref{lem:selection} the
sequentially normal positive functional $F_f$ is normal. 

For general $f_x$, put
$\operatorname{Re}f_x=(f_x+f_x^*)/2$ and
$\operatorname{Im}f_x=(f_x-f_x^*)/(2i)$, and take the normal Jordan
decomposition of each self-adjoint functional. The Jordan components
are positive and their norm sums equal the norm of the self-adjoint
functional. Since
$\|\operatorname{Re}f_x\|,\|\operatorname{Im}f_x\|\leq\|f_x\|$,
all four Jordan components are integrable. This expresses $f$ as a
linear combination of four positive integrable fields, so $F_f$ is
normal too.

The image of \eqref{eq:predual} is norm closed by completeness and the
isometry. It also separates points of $A_\mu$. Indeed, if $[a]\ne0$, put
$\delta=\|[a]\|$ and
$S=\{x:\|a_x\|>\delta/2\}$. The quotient-norm formula implies
$\mu(S)>0$. For $x\in S$, choose a normal functional $f_x$ of norm at
most one such that $|f_x(a_x)|>\delta/2$.  By scaling with a suitable unimodular scalar we may assume that
$f_x(a_x)>\delta/2$. Set $f_x=0$ outside $S$. Then $f\in\mathcal E$ and
$F_f([a])>(\delta/2)\mu(S)>0$.
If this closed subspace were proper in $(A_\mu)_*$, the Hahn--Banach theorem 
would give a nonzero element of $((A_\mu)_*)^*=A_\mu$ annihilating it,
contradicting the point separation above. This proves surjectivity.

Finally, suppose $F_f$ is positive. The involution
$f_x^*(a)=\overline{f_x(a^*)}$ is compatible with
\eqref{eq:predual}.  In particular, injectivity implies $f_x=f_x^*$ almost everywhere.
Modulo a null set, we may write
$f_x=f_x^+-f_x^-$ for the normal Jordan decomposition. Let
$r_x=s(f_x^-)$. The positive and negative parts have orthogonal
supports, so $f_x(r_x)=-\|f_x^-\|$. Hence the projection field $r=(r_x)$ belongs to $N$
and gives
\[
 0\leq F_f([r])=-\int_X\|f_x^-\|\,d\mu(x).
\]
Thus $f_x^-=0$ almost everywhere, completing the converse. The forward
positivity implication was proved above.
\end{proof}

\subsection{Disintegration of a countably additive state}

\begin{theorem}[Product disintegration]\label{thm:product}
With notation as above, let $\psi$ be a state on $N=\oplus_{x \in X} \, N_x$ which is countably additive on
projections, and define $\mu(S)=\psi(1_S)$ for $S \subseteq X$. Then  for all $x \in X$ there exists a normal state
$\psi_x$ on $N_x$ such that
\[
 \psi((a_x))=\int_X\psi_x(a_x)\,d\mu(x), \qquad a \in N.
\]
For this fixed product and fixed measure, the field $(\psi_x)$   of normal functionals is unique up to
almost-everywhere equality.
\end{theorem}
\begin{proof}
The characteristic functions satisfy
$1_{\bigcup_n S_n}=\sum_n1_{S_n}$ for disjoint subsets $S_n$ of $X$.
Countable additivity of $\psi$ therefore makes
$\mu(S)=\psi(1_S)$ a countably additive probability measure
on $\Pset(X)$. We now apply the preceding quotient and
predual constructions to this induced measure $\mu$. In particular, $J_\mu$ and $A_\mu=N/J_\mu$ are
 formed
using this choice of measure. If $a\in J_\mu$ then it is supported on a central projection $1_S$ with
$\mu(S)=0$, namely $S=\{x:a_x\ne0\}$. Since $a=1_S \, a$, the 
Cauchy--Schwarz inequality gives
\[
 |\psi(a)|^2\leq\psi(1_S)\psi(a^*a)=0.
\]
Thus $\psi$ induces a state $\bar\psi$ on $A_\mu$.

This induced state is countably additive on projections. To check this,
lift any orthogonal sequence $(e_n)$ of quotient projections to
coordinate projection fields $p^{(n)}$. To see that such lifts exist,
start with a self-adjoint representative $a^{(n)}$; the identities
$e_n^2=e_n$ imply that $(a_x^{(n)})^2=a_x^{(n)}$ outside a null set;
self-adjointness then makes $a_x^{(n)}$ a projection there, and it may
be replaced by zero on the exceptional set. For $n\ne m$,
$e_ne_m=0$ implies $p_x^{(n)}p_x^{(m)}=0$ outside a null set. There are
only countably many pairs, so their exceptional sets have a null union
$Z$ say. Replacing every $p_x^{(n)}$ by zero for $x\in Z$ makes the fields
pointwise orthogonal without changing their quotient classes. The
coordinatewise strong sum $p_x=\sum_np_x^{(n)}$ represents
$\bigvee_ne_n$: apply Lemma~\ref{lem:quotient} to the increasing finite
partial sums $\sum_{n=1}^k p^{(n)}$. Countable additivity of $\psi$
on the lifted projections gives
\[
 \bar\psi(\bigvee_ne_n)=\psi(p)
 =\sum_n\psi(p^{(n)})=\sum_n\bar\psi(e_n).
\]

By Lemmas~\ref{lem:quotient} and \ref{lem:sigma}, $\bar\psi$ is normal.
Lemma~\ref{lem:predual} gives an integrable positive normal field $f_x$
representing it. For every $S\subseteq X$,
\[
 \mu(S)=\psi(1_S)=\int_S f_x(1)\,d\mu(x).
\]
Thus $\int_S(f_x(1)-1)\,d\mu=0$ for every $S$. Applying this to the
sets where $f_x(1)>1$ and where $f_x(1)<1$ shows that
$f_x(1)=1$ almost everywhere: a nonnegative measurable function has
zero integral only if it vanishes almost everywhere. Modify the field to any
normal state on the exceptional null set and put $\psi_x=f_x$.
For uniqueness, let $(\chi_x)$ be another field of normal states giving
the same representation. Both fields are integrable, since their norms
are one. They induce the same functional on $A_\mu$, so the isometry
in Lemma~\ref{lem:predual} gives
\[
 \int_X\|\psi_x-\chi_x\|\,d\mu(x)
   =\|F_\psi-F_\chi\|=0.
\]
Thus $\psi_x=\chi_x$ outside one null set as functionals, not merely
almost everywhere separately for each chosen operator.
\end{proof}

\section{Proof of sequential \texorpdfstring{weak$^*$}{weak*} continuity}\label{sec:continuity}

For a given sequence, it suffices to work in the von Neumann algebra
it generates. The following well known decomposition (see \cite[Corollary 2.4]{Sh}) puts us directly in
the setting of Theorem~\ref{thm:product}, but with the indexing set $I$ a subset of the unit interval $[0,1]$.

\begin{lemma}[Countably generated reduction]\label{lem:countable-reduction}
Let $N$ be a countably generated von Neumann algebra. Then \[
 N\cong\bigoplus_{t\in I} \, N_t       \qquad I\subseteq[0,1]
\]
where the $N_t$ are nonzero separably acting von Neumann algebras. This is the usual von Neumann algebra direct sum.
Equivalently, there are mutually orthogonal nonzero central projections
$(z_t)_{t\in I}$ in $N$ with supremum $1$ such that $N_t=z_tN$ is separably
acting (that is, representable as a unital von Neumann subalgebra of
$\mathcal B(\ell^2)$). Consequently every $N_t$ is $\sigma$-finite. 
\end{lemma}

Thus, for the continuity argument, after replacing the ambient algebra
by the algebra generated by the sequence, we may assume it is an
$\ell^\infty$-direct sum, indexed by a subset of $[0,1]$, of
separably acting $\sigma$-finite algebras. 

\begin{proof}[Proof of Theorem~\ref{thm:main}]
Assume (i), and let $a_n\to0$ weak$^*$ in $M$.  By the uniform boundedness principle 
$(a_n)$ is bounded, with constant bound $K$ say. Put
\[
 M_0=W^*(1,a_1,a_2,\ldots)\subseteq M.
\]
Of course $a_n\to0$ weak$^*$
in $M_0$. Also, $\varphi|_{M_0}$ is countably additive on projections,
since orthogonal projection sums in $M_0$ agree with those in $M$.

Apply Lemma~\ref{lem:countable-reduction} to $M_0$ and write
$M_0=\oplus_{t \in I} \, M_t$, where $M_t=z_tM_0$ and $I\subseteq[0,1]$.
By Theorem~\ref{thm:product}, the probability measure
\[
 \mu(S)=\varphi\left(\sum_{t\in S}z_t\right)
 \qquad(S\subseteq I)
\]
is countably additive on $\mathcal P(I)$, and there are normal
states $\psi_t$ on $M_t$ such that
\[
 \varphi(b)=\int_I\psi_t(z_tb)\,d\mu(t)
 \qquad(b\in M_0).
\]
For each fixed $t$, the coordinate map $b\mapsto z_tb$ is normal.
Thus $\psi_t(z_ta_n)\to0$, while
$|\psi_t(z_ta_n)|\leq\|a_n\|\leq K$. 
The usual  dominated convergence theorem therefore gives
\[
 \varphi(a_n)=\int_I\psi_t(z_ta_n)\,d\mu(t)\longrightarrow0.
\]
Applying this to $a_n-a$ proves (iii) for an arbitrary weak$^*$ limit
$a$.

For (iii)$\Rightarrow$(ii), a bounded increasing positive sequence
converges weak$^*$ to its supremum. For (ii)$\Rightarrow$(i), apply
sequential normality to the increasing finite partial sums of an
orthogonal projection sequence.

Finally, in a faithful normal concrete representation, a
weak-operator-convergent sequence is norm bounded by uniform
boundedness. On norm-bounded subsets the weak operator and weak$^*$
topologies agree. Thus weak operator convergence of a sequence implies
weak$^*$ convergence, and (iii) gives sequential weak operator
continuity of $\varphi$. Strong operator convergence implies weak
operator convergence, giving the remaining assertion 
(the converse being obvious)
\end{proof}

\begin{remark}[Positive functionals]
The 
 equivalences in Theorem~\ref{thm:main} also clearly hold for every bounded
(positive) functional. 
\end{remark}

\smallskip

Chetcuti, Hamhalter and Weber show in \cite[Theorem 4.3(ii)]{CHW} that  for a \(\sigma\)-finite von Neumann algebra, a sequence in \(M_{\mathrm{sa}}\) converges \(\sigma\)-strongly exactly when it converges in the sequential order topology. Theorem 1.1 here concerns sequential weak* continuity of countably additive states and imposes no \(\sigma\)-finiteness assumption.

The extension step in the following result is a known 
Gleason type theorem for von Neumann algebras \cite{BW}. 

\begin{corollary}[Gleason corollary]
\label{cor:gleason-sequential}
Suppose that a von Neumann algebra $M$ has no nonzero direct summand of type $I_2$.  Let
\[
              F:\operatorname{Proj}(M)\longrightarrow[0,1]
\]
satisfy $F(0)=0$, $F(1)=1$, and
\[
 F\left(\sum_{n=1}^{\infty}p_n\right)
     =\sum_{n=1}^{\infty}F(p_n)
\]
for every mutually orthogonal sequence $(p_n)$ of projections.  Then
$F$ has a unique extension to a state $\varphi$ on $M$, and this state
is sequentially weak$^*$ continuous.
\end{corollary}
\begin{proof}
The hypothesis on $M$ means precisely that there is no nonzero central
projection $z\in M$ for which $zM$ is of type $I_2$.  Since $F$ takes
values in $[0,1]$ and is countably additive, it is in particular a
bounded finitely additive measure on the projection lattice of $M$.
By the Bunce--Wright theorem \cite[Theorem~B]{BW},
there is a unique bounded linear functional $\varphi$ on $M$
whose restriction to the projections is $F$; thus
\[
    \varphi(p)=F(p)
    \qquad \bigl(p\in\operatorname{Proj}(M)\bigr).
\]
This extension is positive.  Indeed if $a\in M_+$, then finite
spectral sums with nonnegative coefficients converge in norm to $a$ (as in the lines before Proposition \ref{prop:abelian-integral-criterion}, or in its proof).
On such a sum $b=\sum_{j=1}^k\lambda_jp_j$, where the $p_j$ are
mutually orthogonal projections and $\lambda_j\geq0$, one has
\[
                 \varphi(b)=\sum_{j=1}^k\lambda_jF(p_j)\geq0.
\]
Norm continuity then gives $\varphi(a)\geq0$.  Moreover,
$\varphi(1)=F(1)=1$, so $\varphi$ is a state.  The same spectral
approximation also shows directly that a bounded linear functional is
determined by its values on projections, which gives the uniqueness.

Finally, 
this state is countably additive on projections by the displayed countable-additivity identity in the statement.
Theorem~\ref{thm:main} now implies that $\varphi$ is
sequentially weak$^*$ continuous.
\end{proof}

\begin{remark} 1)\ 
Only bounded finite additivity is needed for the Bunce--Wright extension
step. 

2)\  There is of course a large literature on Gleason's problem, even in von Neumann algebras.   See e.g.\ \cite{Hamhalter}
 for a summary of some of this.  Unfortunately some of the Physics  literature on Gleason's theorem sometimes contains unsupported statements.   We have not been able to find a correct published proof of results such as Corollary
\ref{cor:gleason-sequential} beyond special cases such as the  well known $\sigma$-finite case, or cf.\ e.g.\  \cite{Mat} for the semifinite algebra case. 
\end{remark}

\section{The weight case} \label{weigh}

The weight case of our main earlier results  runs into Haagerup's big open problem of whether normality of weights is 
commutatively determined.
This is Problem 1.11 in   \cite{Haag}, which we resolve in the Appendix. (This may be moved to  a companion paper.)  
using several of the results proved here. 
This section develops normality and sequential normality criteria for
weights, together with normal-minorant constructions and GNS
closability arguments.  We will also apply some of these results in \cite{BNcm}. 
 
 Unfortunately it seems that there are significant obstacles  to any attempt to extend the integral representation theorem which was key to our solution to Theorem \ref{thm:main} to weights (even weights which we know in advance are $\sigma$-normal say).   Nonetheless we shall still give several applications of Theorem \ref{thm:main} to weights. 

Throughout, $\Phi:M_+\to[0,\infty]$ is a weight, that is, we are assuming that $\Phi$ is additive, positively
homogeneous (hence monotone), with $\Phi(0)=0$.   Unless another qualification is specified, a weight $\Phi$ is called
\emph{semifinite} when
\[
 \mathfrak m_\Phi
 =\operatorname{span}\{x\in M_+:\Phi(x)<\infty\}
\]
is weak* dense in $M$. There are several equivalent characterizations of semifiniteness that the reader can find in the texts cited in our bibliography. We distinguish this ordinary semifiniteness
from the finite-subprojection property (P) and the finite-value
recovery property ® introduced below. In particular, ordinary
semifiniteness does not by itself imply either of those properties.

 Write
\begin{align*}
{\rm  (H)} \quad&\Phi|_{A_+}\text{ is normal for every unital abelian von Neumann subalgebra} \, A \subseteq M,\\
{\rm  (H_\sigma)} \quad&\Phi|_{A_+}\text{ is $\sigma$-normal for every such } A.
\end{align*}
Note that (H) (resp.\ (H$_\sigma$)) implies  that $\Phi$ is completely (resp.\ countably) additive on projections.

\subsection{The abelian case}  \label{abcas} To understand just a few of the new issues arising in the weight case we first look at the abelian case.   Or rather the abelian case of the appropriate related question, since Haagerup's problem is formulated so as to be trivial in  the abelian case, referring to weights which are normal on every 
abelian von Neumann subalgebra.  

 Let $\mu$ be a localizable measure and let $\Phi$ be a weight on
$M=L^\infty(X,\Sigma,\mu)$ which is countably additive on projections.
Then $\nu(E)=\Phi(\chi_E)$ defines a countably additive measure
$\nu\ll\mu$. Put $I_\nu(f)=\int f\,d\nu$. 
The finite-spectrum approximation hypothesis for a weight $\Phi$ on $M$  is
\begin{equation}\label{eq:FS}\tag{FS}
 \Phi(x)=\sup\{\Phi(s):0\le s\le x,\ s\text{ has finite spectrum}\},
 \qquad x\in M_+.
\end{equation}
The elements $s$ in \eqref{eq:FS} are not required to commute with $x$.
We will also use a stronger variant 
\begin{equation}\tag{FS$_{\rm c}$}\label{eq:FSc}
\Phi(x)=\sup\{\Phi(s):0\le s\le x,\ s\in W^*(1,x),\ s\text{ has finite spectrum}\}.
\end{equation}
This implies \eqref{eq:FS}. 
Every
$\sigma$-normal weight, in particular every normal weight, satisfies
\eqref{eq:FSc} and  \eqref{eq:FS}, as may be seen by the functional calculus
for $x$ applied to the usual increasing simple functions $s_n$, so that $s_n(x)  \uparrow x$.

\begin{proposition}[The abelian integral and sigma-normality criterion]
\label{prop:abelian-integral-criterion}
Let $\mu$ be localizable and let $\Phi$ be a weight on
$M=L^\infty(X,\Sigma,\mu)$ that is countably additive on projections.
With the notations above  the following conditions are equivalent:
\[
 ({\rm FS}) \,  ({\rm or} \, ({\rm FS}_{\rm c}))\quad\Longleftrightarrow\quad
 \Phi=I_\nu\quad\Longleftrightarrow\quad
 \Phi\text{ is $\sigma$-normal}\quad\Longleftrightarrow\quad
 \Phi\text{ is norm lower semicontinuous}.
\]
\end{proposition}
\begin{proof} 
As in the proof of \cite[Lemma~3.2]{BGL}, countable additivity  on projections gives: countable additivity of $\nu$, 
$\nu\ll\mu$, $I_\nu$ is well defined on $M_+$, and $I_\nu = \Phi$ 
 on positive simple functions. Taking a 
supremum of positive simple functions $s \leq f$ we have $I_\nu(s) = \Phi(s) \le\Phi(f)$, and $I_\nu(f)\le\Phi(f)$.
Indeed that proof shows $I_\nu = \Phi$ is equivalent to $\sigma$-normality of $\Phi$.

Assume $({\rm FS}_{\rm c})$. For fixed $f\ge0$, the usual lower dyadic
approximants $s_n$ are positive simple
functions in $W^*(1,f)$, with $s_n\uparrow f$. Hence the usual
simple-function formula for the integral gives
\[
 \Phi(f)
 =\sup_{\substack{0\le s\le f\\s\in W^*(1,f),\ s\text{ simple}}}\Phi(s)
 =\sup_{\substack{0\le s\le f\\s\in W^*(1,f),\ s\text{ simple}}}I_\nu(s)
 =I_\nu(f).
\] A similar proof shows that $({\rm FS})$ implies $I_\nu = \Phi$.  Conversely, $\sigma$-normality applied
to these dyadic approximants $(s_n)$ gives $({\rm FS}_{\rm c})$.

Finally, $I_\nu$ is norm lower semicontinuous by Fatou's lemma:
norm convergence in $M$ gives representatives converging pointwise off a common
$\mu$-null, hence $\nu$-null, set.  In the other direction, norm lower
semicontinuity applied to a sequence of increasing simple functions with $s_n\to f$ in norm, together with
monotonicity, gives $$\Phi(f) \leq \lim_n \, \Phi(s_n) \leq \sup_n\Phi(s_n) \leq \Phi(f).$$   Therefore we have 
equality in this string, and therefore have $({\rm FS}_{\rm c})$.
\end{proof}

We can add the following to the equivalences in the last result.  It is the sequential variant of the equivalence 
with (3) in the main theorem of \cite{Haag}.

\begin{lemma}\label{lem:abelian-sequential-lsc}
A weight $\Phi$ on an abelian von Neumann algebra $M$ is
$\sigma$-normal if and only if it is sequentially weak$^*$
lower semicontinuous on $M_+$.
\end{lemma}

\begin{proof}
Sequential weak$^*$ lower semicontinuity implies
$\sigma$-normality by monotonicity, since a bounded increasing
sequence converges weak$^*$ to its supremum.

Conversely, suppose that $\Phi$ is $\sigma$-normal, and let
$a_n\to a$ weak$^*$ in $M_+$. We must show that
$\Phi(a)\le\liminf_n\Phi(a_n)$. This is automatic if the
liminf is infinite. Otherwise, passing to a subsequence,
we may assume that
\[
    \Phi(a_n)\longrightarrow L<\infty,
    \qquad \sup_n\Phi(a_n)<\infty.
\]
The sequence $(a_n)$ is norm bounded, so $b=\sum_{n=1}^{\infty}2^{-n}a_n$ 
converges in norm. By $\sigma$-normality, $\Phi(b)=\sum_{n=1}^{\infty}2^{-n}\Phi(a_n)<\infty.$ 
Let $p=s(b)$ and $q_k=1_{[1/k,\infty)}(b)$. Then
$q_k\uparrow p$ and $\Phi(q_k)\le k\Phi(b)<\infty.$ 
Since $a_n\le 2^n b$, we have $pa_n=a_n$, and weak$^*$
convergence gives $pa=a$.

For each $k$, the map
\[
    \omega_k(y)=\Phi(q_k y),\qquad y\in M_+,
\]
extends to a bounded positive $\sigma$-normal functional
on $M$. Theorem~\ref{thm:mainc}, applied to its associated
finite measure, shows that $\omega_k$ is sequentially
weak$^*$ continuous. Consequently,
\[
    \Phi(q_k a)
      =\lim_{n\to\infty}\Phi(q_k a_n)
      \le\lim_{n\to\infty}\Phi(a_n)=L.
\]
Finally, commutativity gives $q_k a\uparrow a$, so
$\sigma$-normality yields $\Phi(a)=\sup_k\Phi(q_k a)\le L.$
\end{proof}

We now discuss the additional step from $\sigma$-normality to normality:

\begin{corollary} \label{isno} Let $\mu$ be localizable and let $\Phi$ be a weight on
$M=L^\infty(X,\Sigma,\mu)$ satisfying  any one of the equivalent conditions
in Proposition {\rm \ref{prop:abelian-integral-criterion}}.
Then $\Phi$ is normal if and only if $\Phi$  is completely additive on projections.  \end{corollary} 

\begin{proof} 
We have $\Phi=I_\nu$. Choose a maximal orthogonal family of nonzero
$\Phi$-zero projections and let $q$ be its sum. Complete 
additivity on projection gives $\Phi(q)=0$. If $\Phi(e)=0$ for a projection $e$,
then, by commutativity, $e(1-q)$ is a zero-weight projection orthogonal
to the whole family. Maximality forces $e(1-q)=0$, so $e\le q$.
Consequently if $1-q=\chi_S$ then $S$ represents a $\mu$-support for $\nu$:
$\nu$ vanishes off $S$ and is strictly positive on every measurable
$F\subseteq S$ with $\mu(F)>0$. By \cite[Lemma~4.3(1)]{BGL},
$I_\nu=\Phi$ is normal.   The converse is obvious (or see the start of \cite[Section 3]{BGL}). \end{proof}

Because of Haagerup's example in  \cite[Remark 1.12]{Haag}, a (semifinite) weight which is completely additive on projections, even in the abelian case $M = l^\infty$, need not be $\sigma$-normal, let alone normal. 
Thus we cannot replace the equivalent conditions in Proposition {\rm \ref{prop:abelian-integral-criterion}} or the
corollary with semifiniteness (or `orthogonal' or `strong' semifiniteness \cite{GLbook1}).  However there are some  easy solutions of Haagerup's  problem 1.11 and its sequential variants in special cases.

\subsection{Back to the general case}

Now that we understand the abelian case, we consider weights  on a von Neumann algebra $M$ that are countably (resp.\ completely) additive on projections, indeed which are $\sigma$-normal (resp.\ normal) on abelian von Neumann subalgebras.   
The second of these is Haagerup's Problem 1.11.

Considering the abelian example above, one is led to the following variant of Haagerup's question.

\smallskip 

{\em Question:} Is a weight on a von Neumann algebra $\sigma$-normal provided its restriction to every commutative von Neumann subalgebra is $\sigma$-normal (equivalently, by Lemma~\ref{lem:abelian-sequential-lsc},
sequentially weak* lower semicontinuous)?

\begin{proposition} 
\label{prop:H-projection-spectral}
For any weight on any von Neumann algebra,
\begin{align*}
 (H_\sigma)&\Longleftrightarrow
 \text{countably  additive on projections}+({\rm FS}_c),\\
 (H)&\Longleftrightarrow
 \text{completely additive on  projections}+({\rm FS}_c).
\end{align*}
\end{proposition}
\begin{proof}
The forward implications are obvious from the above, and the fact that $({\rm FS}_c)$ is
determined on abelian von Neumann subalgebras.   Conversely, restrict the
weight to a unital abelian von Neumann subalgebra $A$. Additivity on projections is a hereditary property. 
Similarly $({\rm FS}_c)$ is  hereditary too, since
$W^*(1,x)\subseteq A$ for $x\in A_+$. Represent $A$ as an $L^\infty$
space of a localizable measure. 
Proposition~\ref{prop:abelian-integral-criterion} gives
$\sigma$-normality under countable additivity on projections.
Corollary \ref{isno} gives normality under complete
projection additivity. \end{proof}

The following criterion is related to `strict semifiniteness'.  
We say that a weight $\Phi$ on a von Neumann algebra $M$ 
has an \emph{strict decomposition} or is {\em strictly semifinite} if there are bounded positive functionals
$(\rho_i)_{i\in I}\subset M^{*+}$ and mutually orthogonal projections
$(e_i)_{i\in I}$ such that
\begin{equation} \label{eq:SD-three}
 \sum_i e_i=1,\qquad \rho_i(1-e_i)=0,\qquad
 \Phi(x)=\sum_i\rho_i(x)\quad(x\in M_+).
\end{equation}

It is well known that for $x\in M_+$, positivity and $\rho_i(1-e_i)=0$ imply
$\rho_i(x)=\rho_i(e_i x e_i)$.  (This follows for example from Choi's multiplicative domain result.)
Moreover, if $j\ne i$, then
$e_i\le1-e_j$, so $\rho_j(e_i)=0$.  Consequently
\[
 \Phi(e_i)=\rho_i(e_i)=\rho_i(1)<\infty,
 \qquad
 \Phi(e_i x e_i)=\rho_i(x),
\]
and hence
\begin{equation} \label{eq:SD-corners-three}
 \Phi(x)=\sum_i\Phi(e_i x e_i),\qquad x\in M_+.
\end{equation}

For comparison, in the usual normal-weight setting, strict
semifiniteness is expressed by a sum of bounded normal positive
functionals with mutually orthogonal support projections. Such a sum
gives the decomposition above, after adjoining the complementary
zero-weight projection if necessary.

\begin{theorem}[Strict decomposition criteria]\label{thm:three-strict-decomposition}
Let $\Phi$ be a weight on a von Neumann algebra $M$ with a strict decomposition 
as in {\rm (\ref{eq:SD-three})}.
\begin{enumerate}[label=\textup{(\roman*)}]
\item If $\Phi$ is countably additive on projections, then it is
$\sigma$-normal.
\item If $\Phi$ is completely additive on projections, then it is normal.
\end{enumerate} \end{theorem}
\begin{proof}
Put $\lambda_i=\Phi|_{(e_iMe_i)_+}$.  Since
$\Phi(e_i)<\infty$, $\lambda_i$ extends uniquely to a bounded positive
functional on $e_iMe_i$ \cite{Takesaki}. 
In (i), $\lambda_i$ inherits countable projection additivity from
$\Phi$.   Hence every $\lambda_i$ is $\sigma$-normal by Theorem \ref{thm:main}.  
Since $x\mapsto e_i x e_i$ is normal and positive,
\[
 \omega_i(x):=\lambda_i(e_i x e_i)
\]
is $\sigma$-normal on $M$.  By \eqref{eq:SD-corners-three},
$\Phi=\sum_i\omega_i$.  Therefore, for $0\le x_n\uparrow x$,
\[
 \Phi(x)=\sum_i\sup_n\omega_i(x_n)
        =\sup_n\sum_i\omega_i(x_n)=\sup_n\Phi(x_n),
\]
where both suprema are over nonnegative numbers.  This proves (i).

In (ii) the same argument uses arbitrary orthogonal families: each
$\lambda_i$ is completely additive on projections and hence normal by
the classical characterizarton of normal states.  Thus each
$\omega_i$ is normal.  If $0\le x_\alpha\uparrow x$, then
\[
 \Phi(x)=\sum_i\sup_\alpha\omega_i(x_\alpha)
        =\sup_\alpha\sum_i\omega_i(x_\alpha)=\sup_\alpha\Phi(x_\alpha),
\]
again because finite partial sums and the directed supremum commute for
nonnegative numbers.  Hence $\Phi$ is normal, proving (ii).
\end{proof}

{\bf Remark.} If the weight  $\Phi$ in the last result is faithful and countably additive on projections, 
one may deduce further that it is normal.   Indeed if $(q_j)_{j\in J}$
is any orthogonal family of nonzero projections in $e_iMe_i$,
then
\[
 \lambda_i(q_j)>0,
 \qquad
 \sum_{j\in F}\lambda_i(q_j)\le\lambda_i(e_i)
 \quad(F\subseteq J\text{ finite}).
\]
For each positive integer $m$, only finitely many of these
numbers can be at least $1/m$, so  the set $J$ is countable.
Thus $\lambda_i$ is completely additive on projections and
hence normal.  Now it is easy to see that $\Phi$ is normal \cite{Haag}. 

\bigskip

If one assumes that a weight $\Phi$ on a von Neumann algebra $M$ is {\em sequentially weak* lower semicontinuous} on
$M_+$, then $\sigma$-normality follows immediately.  Indeed, if $0\le x_n\uparrow x$ then $x_n\to x$ strongly, 
hence weak*.  Since $\Phi$ is monotone, $\Phi(x_n)\uparrow L\le \Phi(x)$.
On the other hand, sequential weak* lower semicontinuity gives
$\Phi(x)\le \liminf_{n\to\infty}\Phi(x_n)=L.$
Therefore $\Phi(x_n)\uparrow \Phi(x).$ Thus sequential weak*  lower semicontinuity on $M_+$ implies
$\sigma$-normality.

\subsection{Normal projection detection}
Put
\[
\mathcal D_n(\Phi)=\{\omega\in M_*^+:\omega\le\Phi\}.
\]
We say that $\Phi$ has \emph{normal projection detection} or has property $(F^n)$, if
\begin{equation}\tag{$F^n$}\label{eq:Fn-three}
\Phi(p)=\sup_{\omega\in\mathcal D_n(\Phi)}\omega(p),\qquad p\in\operatorname{Proj}(M).
\end{equation}

\begin{theorem} 
\label{thm:normal-projection-detection}
Let $\Phi$ be a weight on any von Neumann algebra. The following are
equivalent:
\begin{enumerate}[label=\textup{(\roman*)}] 
\item $\Phi$ is normal.
\item $\Phi$ satisfies $(\mathrm{FS}_{\rm c})$ and  $(F^n)$.
\item $\Phi$ satisfies $(\mathrm{FS})$ and $(F^n)$.
\end{enumerate} \end{theorem}

\begin{proof} We already observed after its definition that $\sigma$-normality implies $({\rm FS}_{\rm c})$.
Normality implies 
$(F^n)$ by (4) in Haagerup's main theorem in \cite{Haag}.
Since $({\rm FS}_{\rm c})$ implies $(\mathrm{FS})$, it remains to
prove that normality follows from  $(\mathrm{FS})+(F^n)$.
Define $\Theta(x)=\sup_{\omega\in\mathcal D_n(\Phi)}\omega(x)$, which preserves increasing suprema.  Clearly
$\Theta\le\Phi$.  
We first prove equality on every positive
finite-spectrum element.
Write
\[
 s=\sum_{j=1}^m\alpha_jp_j,
 \qquad \alpha_j>0,
\]
where the $p_j$ are mutually orthogonal projections.  Suppose first that $\Phi(s)<\infty$.  If
$p=\supp(s)=\sum_jp_j$ and $\delta=\min_j\alpha_j$, then
$\delta p\le s$, so $\Phi(p)\le\delta^{-1}\Phi(s)<\infty.$ 
For $\omega\in\mathcal D_n(\Phi)$, both restrictions to $pMp$ are
bounded positive functionals and $\omega|_{pMp}\le\Phi|_{pMp}$.
Since $0\le s\le\|s\|p$,
\[
 0\le \Phi(s)-\omega(s)
 \le \|s\|\bigl(\Phi(p)-\omega(p)\bigr).
\]
Condition ($F^n$) at the finite-weight projection $p$ therefore
gives $\Theta(s)=\Phi(s)$.

Now suppose that $\Phi(s)=\infty$.  Since $s$ is a finite positive
linear combination of the $p_j$, at least one $p_j$ has infinite
weight.  For such a $j$, condition ($F^n$) gives
$\sup_{\omega\in\mathcal D_n(\Phi)}\omega(p_j)=\infty.$ 
Hence $\Theta(s)\ge
 \alpha_j\sup_{\omega\in\mathcal D_n(\Phi)}\omega(p_j)=\infty,$ 
and again $\Theta(s)=\Phi(s)$.

Fix $x\in M_+$. By $(\mathrm{FS})$,
\[
 \Phi(x)
 =\sup_{\substack{0\le s\le x\\s\text{ finite-spectrum}}}\Phi(s)
 =\sup_s\Theta(s)
 \le\Theta(x)\le\Phi(x).
\]
Thus $\Phi = \Theta$ on $M_+$, so that $\Phi$ is normal. 
\end{proof}

\begin{corollary} \label{cor:haagerup-projection-detection}
Haagerup's Problem~1.11 is equivalent to asking whether  {\rm (H)} implies
$(F^n)$ for every weight.
\end{corollary}

\begin{proof}
Condition (H) implies $({\rm FS}_{\rm c})$ since the latter is
determined on abelian von Neumann subalgebras.  So the previous theorem applies.
\end{proof}

The implication $(H)\Rightarrow(F^n)$ is established in
\cite{Bd} through the general normality theorem. Thus the
reformulation in Corollary~{cor:haagerup-projection-detection} has an affirmative answer.
 Theorem~\ref{thm:normal-projection-detection} 
remains an independent projection-detection criterion, including its
formulation using $(\mathrm{FS})$ rather than
$(\mathrm{FS}_{\mathrm c})$.

Of course every normal weight has (H) and $(F^n)$.

\medskip

{\bf Remarks.} 1)\ In fact, even $(H_\sigma)+(F^n)$ implies normality, since
(H$_\sigma$) already implies $({\rm FS}_{\rm c})$.   

\smallskip

2)\ Put
\[
 \mathcal D_\sigma(\Phi)
 =\{\omega\in M^{*+}:\omega\text{ is }\sigma\text{-normal},
                         \ \omega\le\Phi\}.
\]
The sequential projection-detection condition $(F^\sigma)$ is
\[
 \Phi(p)=\sup_{\omega\in\mathcal D_\sigma(\Phi)}\omega(p)
 \qquad(p\in\operatorname{Proj}(M)).
\]
The proof of Theorem~6.6, with increasing sequences in place of nets,
shows that $(\mathrm{FS})+(F^\sigma)$ implies $\sigma$-normality.
 Indeed
(FS) and $(F^\sigma)$ imply $\sigma$-normality,
without any projection-additivity assumption.

\bigskip

We now consider two variants of semifiniteness of weights.  We say that a weight has (finite) {\em recovery}, or satisfies (R), if  
\[ \Phi(x)=\sup\{\Phi(y):0\le y\le x,\ \Phi(y)<\infty\},
\qquad x\in M_+.  \tag{R} \]
This variant of  semifiniteness is considered e.g.\ in  II.6.7.5 in \cite{Black}.

We will say that a weight is {\em strongly semifinite} (see \cite{GLbook}, this is sometimes called densely semifinite), if 
\[
 0\ne e\in\operatorname{Proj}(M)\Longrightarrow
 \exists\,0\ne f\le e:\quad f\in\operatorname{Proj}(M),\quad
 \Phi(f)<\infty
 \tag{P}
\]
Since this property will be used a lot here we will sometimes simply refer to  it as the (P)
property.  (The `P' standing for `projection'.)  Strong semifiniteness is equivalent to many other interesting conditions, but we shall
not need these other formulations here. For example, it is very easy to see that it is equivalent to:
for every $0\ne x\in M_+$ there exists
$0\ne y\in M_+$ such that $y\le x$ and $\Phi(y)<\infty$.
For a weight $\Phi$, we recall (from e.g.\ \cite{Takesaki2} or  \cite[II.6.7.3]{Black}) the notation $\mathfrak n_\Phi$ for the left ideal
$\{a\in M:\Phi(a^*a)<\infty\}.$ The weight extends linearly
to the span of its finite-weight positive elements. Polarization gives
the GNS construction
\[
 (\pi_\Phi,H_\Phi,\Lambda_\Phi),\qquad
 \pi_\Phi(b)\Lambda_\Phi(a)=\Lambda_\Phi(ba),\qquad
 \|\Lambda_\Phi(a)\|^2=\Phi(a^*a).
\]
The range of $\Lambda_\Phi$ is dense in $H_\Phi$, and $\pi_\Phi$ is
contractive. 

\subsection{Strong semifiniteness and property {\rm (R)}} 

\begin{lemma}[GNS closedness and lower semicontinuity]
\label{lem:closed-GNS}
For a weight $\Phi$ on a von Neumann algebra $M$, consider
$\Lambda_\Phi:\mathfrak n_\Phi\subseteq M\to H_\Phi$, with the norm
topologies on $M$ and $H_\Phi$.
\begin{enumerate}[label=\textup{(\roman*)}]
\item If $\Phi$ is $\sigma$-normal, then $\Lambda_\Phi$ is closed.
\item If $\Lambda_\Phi$ is closable, then $\Phi$ is norm lower
semicontinuous at every $x\in M_+$ with $\Phi(x)<\infty$.
\item If $\Lambda_\Phi$ is closed, then $\Phi$ is norm lower
semicontinuous on all of $M_+$.
\end{enumerate}
\end{lemma}
\begin{proof}
{\rm (i)}\ Suppose $a_n\to a$ in norm and $\Lambda_\Phi(a_n)\to\eta$ in Hilbert
space norm. Choose a subsequence $(b_k)$ sufficiently rapidly convergent
in both norms so that, for $d_k=b_{k+1}-b_k$,
\[
 \|d_k\|\le2^{-k},\qquad \|\Lambda_\Phi(d_k)\|\le2^{-k}.
\]
For $m\ge1$ the series
\[
 B_m=\sum_{k=m}^\infty 2^k d_k^*d_k
\]
converges in norm, since the norm of its $k$th summand is at most
$2^{-k}$. Sigma-normality,
followed by additivity on finite sums, gives
\[
 \Phi(B_m)=\sum_{k=m}^\infty2^k\|\Lambda_\Phi(d_k)\|^2
 \le\sum_{k=m}^\infty2^{-k}=2^{1-m}.
\]
We have $$(a-b_m)^*(a-b_m)
 =\left(\sum_{k=m}^\infty d_k\right)^*
   \left(\sum_{k=m}^\infty d_k\right).$$ 
   For operators \(a_1,\ldots,a_N\) and positive real numbers \(w_1,\ldots,w_N\), we have 
$$\left(\sum_{k=1}^{N}a_k\right)^*
\left(\sum_{k=1}^{N}a_k\right)
\le
\left(\sum_{k=1}^{N}w_k^{-1}\right)
\left(\sum_{k=1}^{N}w_k\,a_k^*a_k\right).
$$
To see this note that for every vector \(\xi\) we have 
\[
\begin{aligned}
\left\|\sum_k a_k\xi\right\|^2
&\le\left(\sum_k\|a_k\xi\|\right)^2\\
&=\left(\sum_k w_k^{-1/2}\,w_k^{1/2}\|a_k\xi\|\right)^2\\
&\le\left(\sum_k w_k^{-1}\right)
       \sum_k w_k\|a_k\xi\|^2.
\end{aligned}
\]
Taking $a_k = d_k$ and \(w_k=2^k\), and taking norm limits, yields
  $$\left(\sum_{k=m}^\infty d_k\right)^*
   \left(\sum_{k=m}^\infty d_k\right)
 \le\left(\sum_{k=m}^\infty2^{-k}\right)
       \left(\sum_{k=m}^\infty2^k d_k^*d_k\right)
 =2^{1-m}B_m.$$
Consequently
\[
 \Phi((a-b_m)^*(a-b_m))\le2^{2-2m}.
\]
Thus $a-b_m\in\mathfrak n_\Phi$, so $a\in\mathfrak n_\Phi$, and
\[
 \|\Lambda_\Phi(a)-\Lambda_\Phi(b_m)\|\le2^{1-m}\longrightarrow0.
\]
It follows that $\Lambda_\Phi(a)=\eta$. Thus the graph is 
closed.

\smallskip 

\textrm{(ii) and (iii)}\   Let $x_n,x\in M_+$ with $x_n\to x$ in
norm and $L=\liminf_n\Phi(x_n)$.  The lower semicontinuity at $x$ is obvious if $L= \infty$. Otherwise
pass to a subsequence with finite weights converging to $L< \infty$. Set $a_n=x_n^{1/2}$ and $a=x^{1/2}$.
Then $a_n\to a$ in norm and $\Lambda_\Phi(a_n)$ is bounded.
A further subsequence of this converges weakly to some $\eta\in H_\Phi$.
By Mazur's lemma we obtain convex combinations of tails of $(a_n)$, namely 
\[
 b_j=\sum_{n\ge j}\alpha_{j,n}a_n,\qquad
 \alpha_{j,n}\ge0,\qquad\sum_{n\ge j}\alpha_{j,n}=1,
\]
such that $\Lambda_\Phi(b_j)\to\eta$ in norm. Also $b_j\to a$ in norm.

For (ii), the additional assumption $\Phi(x)<\infty$ gives
$a\in\mathfrak n_\Phi$. Closability applied to $b_j-a$ yields
$\eta=\Lambda_\Phi(a)$. For (iii), closedness itself gives both
$a\in\mathfrak n_\Phi$ and $\eta=\Lambda_\Phi(a)$. In either case, weak lower semicontinuity of the Hilbert
norm gives
\[
 \Phi(x)=\|\eta\|^2
 \le\liminf_n\|\Lambda_\Phi(a_n)\|^2=L,
\]
establishing lower semicontinuity. \end{proof}

\begin{theorem}
\label{thm:sigma-norm-lsc}
Every $\sigma$-normal weight on a von Neumann algebra is norm lower
semicontinuous on the positive cone. 
\end{theorem}
\begin{proof}  Apply parts (i) and (iii) of Lemma \ref{lem:closed-GNS}.
\end{proof}

\begin{corollary}
\label{cor:three-regularities}
For every weight on every von Neumann algebra,
\[
 \begin{gathered}
 \text{sequential weak$^*$ lower semicontinuity}\\
 \Downarrow\\
 \sigma\text{-normality}\\
 \Downarrow\\
 \text{norm lower semicontinuity}.
 \end{gathered}
\]
Norm lower semicontinuity also implies property (FS$_{\rm c}$) and 
\eqref{eq:FS}.
\end{corollary}
\begin{proof}
If $0\le x_n\uparrow x$, then $x_n\to x$ weak$^*$, so weak* lower
semicontinuity and monotonicity give
$\Phi(x)\le\liminf_n\Phi(x_n)\le\Phi(x)$.
The second implication is Theorem~\ref{thm:sigma-norm-lsc}.
For the last assertion, by functional calculus choose positive finite-spectrum `simple functions'  $s_n \in W^*(1,x)$ 
with $s_n \le x$ and $\|s_n-x\|\to0$. Then
\[
 \Phi(x)\le\liminf_n\Phi(s_n)\le\Phi(x).
\]
This gives \eqref{eq:FS}.
\end{proof}

Let $S_+(M)$ 
be the bounded positive $\sigma$-normal functionals on $M$.
 Pedersen--Takesaki \cite[Corollary~7.3, p.~82]{PedersenTakesaki}
prove that every norm-lower-semicontinuous weight on a $C^*$-algebra
is a sum of bounded positive functionals. Applied to $M$ regarded as
a $C^*$-algebra, this gives
\begin{equation}\label{eq:norm-lsc-sum}
 \Phi(x)=\sum_{i\in I}\rho_i(x),\qquad x\in M_+,
 \qquad \rho_i\in M^{*+}.
\end{equation}
Here the sum is the supremum of finite subsums. 

\begin{lemma}\label{lem:finite-domain-part}
Let $\Phi$ be any weight and $0\le\rho\le\Phi$ a bounded positive
functional. There is a bounded positive functional $\omega\le\rho$
which agrees with $\rho$ on every positive element of finite
$\Phi$-weight. It can be constructed as follows. In the GNS representation
$(\pi,H,\xi)$ of $\rho$, put
\[
 K=\overline{\{\pi(a)\xi:a\in\mathfrak n_\Phi\}},\qquad P=P_K,
 \qquad \omega(x)=\langle\pi(x)P\xi,P\xi\rangle.
\]
If $\Phi$ is countably additive on projections, then $\omega\in S_+(M)$.
If $\Phi$ is completely additive on projections, then $\omega$ is normal.
\end{lemma}
\begin{proof} 
Since $\mathfrak n_\Phi$ is a left ideal, $K$ is invariant under $\pi(M)$, and hence is reducing. 
Thus $P$ commutes with
$\pi(M)$ and, for $x\ge0$,
\[
 \rho(x)=\omega(x)+
 \langle\pi(x)(1-P)\xi,(1-P)\xi\rangle.
\]
This gives $0\le\omega\le\rho$.
If $y \in M_+$ with $\Phi(y)<\infty$, then $y^{1/2}\in\mathfrak n_\Phi$, so
\[
 \pi(y^{1/2})P\xi=P\pi(y^{1/2})\xi=\pi(y^{1/2})\xi.
\]
Taking squared norms proves 
$\omega(y)=\rho(y)$ if $\Phi(y)<\infty.$

Assume next countable additivity on projections.  Fix $a\in\mathfrak n_\Phi$
and set $p_m=\Chi_{[1/m,\infty)}(a^*a)$. Then
\[
 \Phi(p_m)\le m\Phi(a^*a)<\infty,
 \qquad \|a-ap_m\|^2\le\frac1m.
\]
By Theorem \ref{thm:main}  the bounded functional
$\lambda =\Phi|_{p_mMp_m}$ is $\sigma$-normal. The positive
functional $\rho|_{p_mMp_m}$ is dominated by $\lambda$. It is
therefore $\sigma$-normal (since  $0\le\rho(p_m a_n p_m)\le\lambda(a_n)\to0$ when $a_n\downarrow0$ in $p_m M p_m$). 
Consequently the functional $\rho_m^a(x) = \rho(p_ma^*xa p_m)$ 
is $\sigma$-normal, since $x\mapsto p_ma^*xa p_m$ preserves bounded
increasing sequential suprema. By a simple triangle inequality argument (writing $a^*xa - b^* x b = a^*x(a-b) + (a-b)^*xb$), 
we have 
\[
 \|\rho_{a}-\rho_{b}\|
 \le\|\rho\|(\|a\|+\|b\|)\|a-b\|,
 \qquad\rho_a(x)=\rho(a^*xa).
\]
Thus  the functionals $\rho_m^a$ converge in norm to
$\rho_a(x)=\langle\pi(x)\pi(a)\xi,\pi(a)\xi\rangle$.
Hence $\rho_a$ is $\sigma$-normal.

The vectors $\pi(a)\xi$, $a\in\mathfrak n_\Phi$, are dense in $K$.
For vector functionals of the contractive representation $\pi$, a similar triangle inequality argument gives
\[
 \|\omega_\eta-\omega_\zeta\|
 \le(\|\eta\|+\|\zeta\|)\|\eta-\zeta\|.
\]
It follows, again by norm closedness, that every vector functional from
$K$ is $\sigma$-normal.  Thus $\omega$  is $\sigma$-normal.

If we have complete additivity on projections an almost identical argument, using bounded
increasing nets and the normal versions of the closure and domination
statements, proves that $\omega$ is normal.
\end{proof}

\begin{proposition} 
\label{prop:sigma-finite-values}
Let $\Phi$ be norm lower semicontinuous and countably additive on
projections. There exists a $\sigma$-normal weight $\Theta\le\Phi$,
expressible as a sum of members of $S_+(M)$, such that
\[
 \Theta(y)=\Phi(y)\qquad(y\in M_+,\ \Phi(y)<\infty).
\]
If we have complete additivity on projections, $\Theta$ may instead be expressed
as a sum of bounded normal positive functionals and is normal. \end{proposition}
\begin{proof} Each $\rho_i$ in  \eqref{eq:norm-lsc-sum} is dominated by $\Phi$.
Apply Lemma~\ref{lem:finite-domain-part} to obtain $\omega_i\le\rho_i$
with the stated finite $\Phi$-value agreement and appropriate type of normality.  Put
$\Theta=\sum_i\omega_i$.
Then $\Theta\le\Phi$ and the finite values agree term by term. 
Clearly $\Theta$ preserves bounded increasing sequential suprema (net suprema in the normal case).
\end{proof}

\begin{theorem}
\label{thm:recovery-sigma}
Let $M$ be any von Neumann algebra. Suppose that $\Phi$ is a weight
countably additive on projections and satisfying the finite-value
recovery condition {\rm (R)}. The following conditions are equivalent:
\begin{enumerate}[label=\textup{(\roman*)}]
\item $\Phi$ is $\sigma$-normal.
\item $\Phi$ is norm lower semicontinuous on $M_+$.
\item $\Phi(x)=\sup_{\omega\in\mathcal F}\omega(x)$ for every
$x\in M_+$, for some $\mathcal F\subseteq S_+(M)$.
\item $\Phi(x)=\sum_{i\in I}\omega_i(x)$ for every $x\in M_+$,
for some family $(\omega_i)_{i\in I}\subseteq S_+(M)$.
\end{enumerate}
The families may be arbitrary, and their supports need not be orthogonal. 
\end{theorem}
\begin{proof} 
The implication (i)$\Rightarrow$(ii) is
Theorem~\ref{thm:sigma-norm-lsc}. Assume (ii), and let $\Theta$ be
supplied by Proposition~\ref{prop:sigma-finite-values}. For
$0\le y\le x$ with $\Phi(y)<\infty$,
\[
 \Phi(y)=\Theta(y)\le\Theta(x).
\]
Taking the supremum and using (R) gives $\Phi(x)\le\Theta(x)$.
Since $\Theta\le\Phi$, equality holds, proving (iv).
For (iv)$\Rightarrow$(iii), take the family of finite partial sums.
Finally, (iii) gives, whenever $0\le x_n\uparrow x$,
\[
 \Phi(x)=\sup_{\omega\in\mathcal F}\sup_n\omega(x_n)
        =\sup_n\sup_{\omega\in\mathcal F}\omega(x_n)
        =\sup_n\Phi(x_n).
\]
This proves (i). 
\end{proof}

\begin{corollary}[The corresponding normality criterion]
\label{cor:recovery-normal}
On any von Neumann algebra, a weight which is completely additive on
projections and satisfies recovery {\rm (R)}  is normal if and only
if it is norm lower semicontinuous. In that event it is a sum of bounded
normal positive functionals.
\end{corollary}
\begin{proof} 
For the reverse implication, use the normal version of
Proposition~\ref{prop:sigma-finite-values}  and the  argument giving $\Theta = \Phi$ in
Theorem~\ref{thm:recovery-sigma}. Conversely, normality implies
$\sigma$-normality, hence norm lower semicontinuity by
Theorem~\ref{thm:sigma-norm-lsc}. Equivalently one can use Haagerup's
weak$^*$ lower semicontinuity theorem \cite[Theorem~1.8]{Haag}.
\end{proof}

{\bf Remark.} The hypothesis (R) is used in the preceding proofs to extend
finite-value agreement to all positive elements. The commutative-
determination theorems of \cite{Bd}, together with
Proposition~\ref{prop:H-projection-spectral} and Corollary~\ref{thm:sigma-norm-lsc}, remove (R) from the equivalence
between $\sigma$-normality and norm lower semicontinuity under
countable projection additivity, and from the corresponding normality
equivalence under complete projection additivity. 

\begin{lemma}\label{lem:orthogonal-GNS-series}
Assume {\rm (H$_\sigma$)}. Let $c_k\in\mathfrak n_\Phi$ and let $q_k$ be
mutually orthogonal projections with $c_k=c_kq_k$. If
\[
 \sum_{k=1}^{\infty}2^k
 \bigl(\|c_k\|^2+\|\Lambda(c_k)\|^2\bigr)<\infty,
\]
then $c=\sum_kc_k$ converges in operator norm, belongs to
$\mathfrak n_\Phi$, and
$\Lambda(c)=\sum_k\Lambda(c_k)$ in the Hilbert space norm.
\end{lemma}

\begin{proof} 
A Cauchy--Schwarz argument gives $\sum_k\|c_k\|<\infty$. The series
\[
 B=\sum_{k=1}^{\infty}2^kc_k^*c_k
\]
also converges in operator norm. Its positive summands have orthogonal
supports, and therefore commute. Apply (H$_\sigma$) to the abelian
algebra they generate together with $1$,  to obtain
\[
 \Phi(B)=\sum_{k=1}^{\infty}2^k\Phi(c_k^*c_k)
        =\sum_{k=1}^{\infty}2^k\|\Lambda(c_k)\|^2<\infty.
\]
For finite partial sums, the argument in the proof of Lemma \ref{lem:closed-GNS} gives \[
 \left|\sum_{k=N+1}^{L}c_k\right|^2
 \le\left(\sum_{k>N}2^{-k}\right)
          \sum_{k=N+1}^{L}2^kc_k^*c_k\le2^{-N}B.
\]
Passing to the operator-norm limit preserves this inequality.  Applying $\Phi$, we see as the proof just referred to
that if 
$N=0$ then $c\in\mathfrak n_\Phi$.  For general $N$ it gives
\[
 \left\|\Lambda(c)-\sum_{k=1}^{N}\Lambda(c_k)\right\|^2
 \le2^{-N}\Phi(B)\longrightarrow0,
\] as desired. 
\end{proof}

Lemma  \ref{lem:closed-GNS} (i) assumes $\sigma$-normality and proves closedness. Next we assume only  (H$_\sigma$) and
prove closability. 

\begin{theorem}[Norm closability of the GNS map]
\label{thm:Hsigma-GNS-closable}
If $\Phi$ satisfies {\rm (H$_\sigma$)}, then $\Lambda_\Phi$ is closable
from the operator norm on its domain to the norm on $H_\Phi$.
Explicitly, $a_n\to0$ in operator norm and
$\Lambda(a_n)\to\eta$ in Hilbert-space norm imply $\eta=0$.
\end{theorem}
\begin{proof} 
Suppose $a_n\in\mathfrak n_\Phi$, $a_n\to0$, and
$\Lambda(a_n)\to\eta$. We construct mutually orthogonal projections
$p_k$ of finite weight and elements $b_k=b_kp_k$ such that
\begin{equation}\label{eq:gliding-hump-bounds}
 \|b_k\|\le2^{-k},\qquad
 \|\Lambda(b_k)-\eta\|\le2^{-k}.
\end{equation}
Suppose $p_1,\ldots,p_{k-1}$ are already chosen and put
$P=p_1+\cdots+p_{k-1}$, taking $P=0$ initially. Then
$\Phi(P)<\infty$ and
\[
 \|\Lambda(a_nP)\|^2
 =\Phi(Pa_n^*a_nP)\le\|a_n\|^2\Phi(P)\longrightarrow0.
\]
Choose $n$ with $\|a_n\|\le2^{-k}$ and with both
$\|\Lambda(a_n)-\eta\|$ and $\|\Lambda(a_nP)\|$ at most
$2^{-k}/3$. Set $d=a_n(1-P)=a_n-a_nP \in\mathfrak n_\Phi$.

Put $h=d^*d$, so that $\Phi(h)<\infty$, and for $j\ge1$ let
$q_j=\mathbf{1}_{[2^{-j},\infty)}(h).$ 
Since $d=d(1-P)$, we have $q_j\le1-P$. Moreover,
$2^{-j}q_j\le h$, and hence
\[
 \Phi(q_j)\le 2^j\Phi(h)<\infty.
\]
The elements $hq_j$ increase to $h$ in the abelian von Neumann
algebra $W^*(1,h)$: indeed 
$0\le h-hq_j\le2^{-j}1$ by functional calculus. Thus (H$_\sigma$) gives
$\Phi(hq_j)\uparrow\Phi(h).$ 
Because $q_j$ commutes with $h$ and $\Phi(h)<\infty$, it follows that
\[
 \|\Lambda(d-dq_j)\|^2
 =\Phi\bigl((1-q_j)h(1-q_j)\bigr)
  =\Phi(h)-\Phi(hq_j)\longrightarrow0. \]
Choose $j$ large enough that
$\|\Lambda(d-dq_j)\|\le2^{-k}/3$, and set
$p_k=q_j$ and $b_k=dp_k$.
Then $p_k$ has finite weight and is orthogonal to
$p_1,\ldots,p_{k-1}$, while
$\|b_k\|\le\|d\|\le\|a_n\|\le2^{-k}.$ 
Since $\Lambda(d)=\Lambda(a_n-a_nP)$, the triangle inequality gives
\[ 
 \|\Lambda(b_k)-\eta\|
 \le \|\Lambda(b_k-d)\|
      +\|\Lambda(a_n)-\eta\|
      +\|\Lambda(a_nP)\|
 \le 2^{-k}. 
\]
This proves the required estimates and completes the induction.

The series $X=\sum_{k\ge1}(-1)^{k+1}b_k$ converges absolutely in
operator norm. Set
\[
 d_k=b_{2k-1}-b_{2k},\qquad
 e_k=b_{2k+1}-b_{2k}.
\]
Within each of these two sequences the right supports are dominated by mutually
orthogonal projections (namely,  $p_{2k-1}+p_{2k}$ and $p_{2k}+p_{2k+1}$,
respectively). Moreover,
\[
 \|d_k\|,\ \|\Lambda(d_k)\|\le3 \cdot 2^{-2k},\qquad
 \|e_k\|,\ \|\Lambda(e_k)\|\le\tfrac32 \cdot 2^{-2k}.
\]
Indeed the inequalities here for $d_k, e_k$ follow from the first inequality in (\ref{eq:gliding-hump-bounds}), 
while the  second inequality  there and the triangle inequality gives the  matching inequalities involving $\Lambda$.
Thus Lemma~\ref{lem:orthogonal-GNS-series} applies to both sequences.  We have 
\[
 X=\sum_kd_k=b_1+\sum_ke_k.
\]
Applying $\Lambda$ to these we get similar norm-convergent 
Hilbert-space series. But
\[
 \Lambda(b_1)+\sum_{k=1}^{N}\Lambda(e_k)
       -\sum_{k=1}^{N}\Lambda(d_k)=\Lambda(b_{2N+1})\longrightarrow\eta.
\]
The left side tends to $\Lambda(X)-\Lambda(X)=0$, so $\eta=0$.
\end{proof}

\begin{corollary}[Lower semicontinuity at finite-weight points]
\label{cor:Hsigma-finite-lsc}
Assume {\rm (H$_\sigma$)}. If $x_n,x\in M_+$, $\|x_n-x\|\to0$, and
$\Phi(x)<\infty$, then
\[
 \Phi(x)\le\liminf_n\Phi(x_n).
\]
\end{corollary}
\begin{proof} Theorem~\ref{thm:Hsigma-GNS-closable} gives closability, so 
Lemma~\ref{lem:closed-GNS}(ii) applies. \end{proof}

{\bf Remark.} 
Theorem~\ref{thm:Hsigma-GNS-closable} is strengthened in \cite{Bd}: under
$(H_\sigma)$ the GNS map is closed, not merely closable.
Consequently norm lower semicontinuity holds at every positive
element, without the finite-weight restriction in Corollary~\ref{cor:Hsigma-finite-lsc}. 
The proof of that strengthening uses Lemma~\ref{lem:orthogonal-GNS-series} and 
Theorem~ref{thm:Hsigma-GNS-closable}, so their independent proofs here remain essential.

\begin{proposition}\label{prop:finite-lsc-plus-R}
For a weight $\Phi$ on a von Neumann algebra, norm lower semicontinuity at
every $\Phi$-finite point together with {\rm (R)} implies norm
lower semicontinuity of $\Phi$ on $M_+$.
\end{proposition}

\begin{proof} 
Suppose $x_n\to x$ in norm, with $x_n,x\ge0$. Fix $0\le y\le x$
with $\Phi(y)<\infty$. Douglas factorization in $M$ (see e.g.\ 
\cite[Proposition 6.18]{Mos}) gives a positive
contraction $c\in M$ with $y=x^{1/2}cx^{1/2}$. Put
$y_n=x_n^{1/2}cx_n^{1/2}$. Then $0\le y_n\le x_n$ and $y_n\to y$
in norm by the well known norm continuity of the square root on $M_+$. The assumed lower semicontinuity at $y$ gives
\[
 \Phi(y)\le\liminf_n\Phi(y_n)\le\liminf_n\Phi(x_n).
\]
Now take the supremum over finite-weight $y\le x$, and then use property (R) to obtain our assertion.
\end{proof}

\begin{theorem} \label{affH} Let $\Phi$ be a weight on a von Neumann algebra.  
If {\rm (H)}  holds then {\rm (R)} is equivalent to strong semifiniteness {\rm (P)} for $\Phi$, and implies normality of $\Phi$ and $(F^n)$. 
That is, Haagerup's Problem {\rm 1.11} has an affirmative solution for strongly semifinite weights. \end{theorem}

\begin{proof}
Condition (R) implies (P) for every weight.  Indeed let $e\ne0$ be a projection.
If $\Phi(e)<\infty$, take $f=e$. Otherwise, (R) provides
$0\le y\le e$ with $0<\Phi(y)<\infty$. In particular $y\ne0$, so for
some $\delta>0$ the projection
$f=\Chi_{[\delta,\infty)}(y)$ is nonzero. We have $f\le e$ and
$\delta f\le y$.  Therefore $\Phi(f)\le\delta^{-1}\Phi(y)<\infty$.

Next we show that if $\Phi$ satisfies (H) and (P), then it satisfies (R).
Fix a projection $e$.  By Zorn's lemma choose a maximal orthogonal family
$(f_i)$ of nonzero finite-weight projections below $e$ (there exist such projections by (P)).  If
$\bigvee_i f\, _i<e$, condition (P) supplies another such projection,
a contradiction.  Hence $e=\bigvee_i f_i$.  Since (H) implies 
complete additivity on projections, we have $\Phi(e)=\sum_i\Phi(f_i)$, the supremum of the finite partial sums.  Thus
$$\Phi(e)=\sup\{\Phi(q):q\le e,\ q\text{ a projection},\ \Phi(q)<\infty\}.$$
Now let $\Phi(x)=\infty$ and $K>0$.  By (FS$_{\rm c}$)  on $W^*(1,x)$, choose
$s=\sum_{j=1}^m\alpha_jp_j\le x$ with $\Phi(s)>K$, where the projections here are mutually orthogonal.  If $\Phi(s)<\infty$,
take $y=s$.  Otherwise some $p_j$ has infinite weight, and the last displayed equation 
gives a finite-weight $q\le p_j$
with $\alpha_j\Phi(q)>K$.  Then $y=\alpha_jq\le x$ has
$K<\Phi(y)<\infty$.  This proves (R) (the finite-valued
case is immediate by taking $y=x$).

It remains to prove normality and $(F^n)$. Since (H) implies
(H$_\sigma$), Corollary~\ref{cor:Hsigma-finite-lsc} gives norm
lower semicontinuity at every positive element of finite
$\Phi$-weight. Under either of the equivalent conditions
(R) and (P),  Proposition~\ref{prop:finite-lsc-plus-R}
therefore gives norm lower semicontinuity on all of $M_+$.

Moreover, (H) implies complete additivity on projections:
as usual apply normality on the abelian von Neumann algebra generated
by the family of projections.
Thus Corollary~\ref{cor:recovery-normal}, together with (R),
shows that $\Phi$ is normal. Finally,
Theorem~\ref{thm:normal-projection-detection} gives $(F^n)$.
\end{proof}

\begin{theorem} \label{sv} {\rm (Sequential version.)}\ Let $\Phi$ be a weight on a von Neumann algebra.  
If {\rm (H$_\sigma$)} and  {\rm (R)}  hold then $\Phi$ is $\sigma$-normal. \end{theorem} 

\begin{proof} By an argument  similar to the last proof, 
Corollary~\ref{cor:Hsigma-finite-lsc} and
Proposition~\ref{prop:finite-lsc-plus-R} give norm lower
semicontinuity on $M_+$. Also, (H$_\sigma$) implies countable
additivity on projections, by restriction to the abelian
algebra generated by the countable 
family. Theorem~\ref{thm:recovery-sigma}
therefore gives $\sigma$-normality.
\end{proof}

{\bf Remark.} The direct proofs of Theorems~\ref{affH} and~\ref{sv} use (P) or (R).
In the last section we pt that $(H)$ alone
implies normality and that (H$_\sigma$) alone implies
$\sigma$-normality. (The sequential version may be moved to the present paper.) 
. 

\begin{corollary} \label{cofi} 
Let $\Phi$ be a strongly semifinite weight on a von Neumann algebra $M$
satisfying {\rm (H$_\sigma$)}. If either $\Phi$ is
faithful or $M$ is $\sigma$-finite, then $\Phi$ is normal.
\end{corollary}

\begin{proof}
Condition {\rm (H$_\sigma$)} gives countable additivity on
projections and $(\mathrm{FS}_{\rm c})$. We show that
we have complete additivity on projections.
If $M$ is $\sigma$-finite, every orthogonal family of
nonzero projections is countable, so this is immediate.

If  $\Phi$ is faithful let
$(p_i)_{i\in I}$ be an orthogonal family of projections.
Put $p=\bigvee_i p_i$. If $\sum_i\Phi(p_i)=\infty$,
monotonicity and finite additivity give $\Phi(p)=\infty$.
If the sum is finite, faithfulness implies that
$\Phi(p_i)>0$ whenever $p_i\ne0$. Thus only countably
many $p_i$ are nonzero.  Countable additivity gives
$\Phi(p)=\sum_i\Phi(p_i).$ 
Hence $\Phi$ is completely additive on projections
in either case.
Together with $(\mathrm{FS}_{\rm c})$,
Proposition~\ref{prop:H-projection-spectral} now gives
(H). Since (P) is assumed, Theorem \ref{affH}  implies
that $\Phi$ is normal.
\end{proof}

{\bf Remark.} The Appendix 
removes the
strong-semifiniteness hypothesis from this corollary. If a weight
satisfies (H$_\sigma$) and either the algebra is $\sigma$-finite
or the weight is faithful, then it is normal. 

\bigskip

For a von Neumann algebra $M$ let
$E_+(M)$ denote the bounded positive linear functionals on $M$ that are
sequentially weak$^*$ continuous.  If 
$$\Phi(x)=\sum_{i\in I}\omega_i(x)\quad(x\in M_+),
 \quad \omega_i\in E_+(M),$$ then by considering partial sums we have
 $$\Phi(x)=\sup_{\omega\in\mathcal F}\omega(x)\quad x\in M_+, \; \; \; 
  \mathcal F\subseteq E_+(M).$$
 If $\Phi$ satisfies the latter condition then it is easy to see that 
 $\Phi$ is sequentially weak$^*$ lower semicontinuous on $M_+$.
Indeed if  $x_n\to x$ weak$^*$ in $M_+$, then for every
$\omega\in\mathcal F$,
\[
 \omega(x)=\lim_n\omega(x_n)\le\liminf_n\Phi(x_n).
\]
Then take the supremum over $\omega$.

We give a condition under which the three conditions in the last paragraph are equivalent:

\begin{theorem}[Sequential representations for weights with {\rm (R)}]
\label{thm:SF-representation}
Let $M$ be any von Neumann algebra and let $\Phi$ be a weight on $M$
that  satisfies  
{\rm (R)}. The following conditions
are equivalent:
\begin{enumerate}[label=\textup{(\roman*)}]
\item $\Phi$ is sequentially weak$^*$ lower semicontinuous on $M_+$.
\item There is a family $\mathcal F\subseteq E_+(M)$ such that
\[
 \Phi(x)=\sup_{\omega\in\mathcal F}\omega(x),\qquad x\in M_+.
\]
\item There is a family $(\omega_i)_{i\in I}\subseteq E_+(M)$ such that
\[
 \Phi(x)=\sum_{i\in I}\omega_i(x),\qquad x\in M_+.
\]
\item $\Phi$ is $\sigma$-normal.
\end{enumerate}
The indexing families may be arbitrary. \end{theorem}

\begin{proof} We just discussed the implications 
(iii)$\Rightarrow$(ii)$\Rightarrow$(i).  
Corollary~\ref{cor:three-regularities} gives (i)$\Rightarrow$(iv).
For (iv)$\Rightarrow$(iii), $\sigma$-normality implies countable
additivity on projections. Under (R),
Theorem~\ref{thm:recovery-sigma} gives
\[
 \Phi=\sum_{i\in I}\omega_i,\qquad \omega_i\in S_+(M).
\]
Theorem \ref{thm:main}  gives
$S_+(M)=E_+(M)$, so this is precisely (iii).
\end{proof}

{\bf Remark.}  A faithful normal, strictly semifinite weight can fail (R) (consider  $\Omega(a)=\sum_{n=1}^\infty n^2\langle ae_n,e_n\rangle$ on $B(l^2)$). The weight on \(\ell^\infty(\omega_1)\) which is $0$ if
the support is countable, and otherwise is $\infty$, is sequentially weak* lower semicontinuous and strongly semifinite, but fails (R).
According to an example constructed by Chat-GPT, even  complete additivity on projections plus \((R)\) does not imply \(\sigma\)-normality, sadly. 

\section{Some further special cases of Haagerup's problem and its sequential variant}

Since we understand exactly which weights are $\sigma$-normal  or normal in  the abelian case (see Section 
\ref{abcas}), the reader will  expect that Haagerup's problem and its sequential variant may be solved in the case of finite von Neumann algebras using the faithful normal center valued trace.   We have not included a full
 proof of this here since this would add several pages to our paper.  

\begin{theorem} \label{falgsig} 
Let $M$ be a finite von Neumann algebra and let $\Phi$ be a weight
on $M$. Then $\Phi$ is $\sigma$-normal if and only if it satisfies
(H$_\sigma$).
\end{theorem}

\begin{proof} (Sketch.) 
Let $0\le x_n\uparrow x$.  Using functional calculus for $x$ 
choose increasing finite-spectrum
`simple functions' 
\[
 0\le s_k\uparrow x,\qquad s_k\in W^*(1,x).
\]
One may prove using the center valued trace that for fixed $k$ we have 
\[
 \Phi(s_k)\le \sup_n\Phi(x_n).
\]
Applying (H$_\sigma$) inside $W^*(1,x)$, we obtain
\[
 \Phi(x)=\sup_k\Phi(s_k)
 \le \sup_n\Phi(x_n)\le\Phi(x).
\]
Thus $\Phi$ is $\sigma$-normal. The converse follows by restriction.
\end{proof}

Solving Haagerups's problem 1.11 in the finite algebra case is now not much more difficult.

\begin{corollary}[Haagerup's Problem 1.11 for finite algebras]
\label{cor:haagerup-finite}
Let $M$ be a finite von Neumann algebra and $\Phi$ a weight on $M$.
If $\Phi|_{A_+}$ is normal for every unital abelian von Neumann
subalgebra $A\subseteq M$, then $\Phi$ is normal on $M$. 
\end{corollary}
\begin{proof} (Sketch.) One may choose central projections with  $1=\sum_i z_i$,
and  with each $z_iM$
$\sigma$-finite.
Let $\Phi_i$ be the restriction to $(z_iM)_+$.
The abelian hypothesis passes to $z_i M$ clearly, thus $\Phi_i$ satisfies $(H_\sigma)$
and hence is $\sigma$-normal by the previous result.   It is now easy to argue 
that it is normal, since $z_i M$ is $\sigma$-finite (e.g.\ c.f.\ the 
proof of Lemma \ref{lem:selection}).

For each $x\in M_+$, the algebra $A=W^*(1,x,\{z_i\}_i)$ 
is abelian. Normality of $\Phi$ on $A$, applied to the finite
central partial sums increasing to $x$, gives
\[
 \Phi(x)=\sup_{F\Subset I}\Phi\!\left(\sum_{i\in F}z_ix\right)
        =\sum_i\Phi_i(z_ix).
\]
Each map $x\mapsto\Phi_i(z_ix)$ is a normal weight on $M$.  
Interchanging the suprema over finite subsets and over an increasing
net proves that their sum is normal. Hence $\Phi$ is normal.
\end{proof}

We omit here the proof of the following
striking result, except to say that it can be proved with the help
of  tricks mentioned in the proof of Corollary \ref{cor}.  In particular, reducing to the finite algebra case, and to 
the properly infinite case.  The finite algebra case then needs some of the ideas for the finite algebras above.

\begin{theorem} \label{falgcon} 
Let \(M\) be a von Neumann algebra and \(\Phi:M_+\to[0,\infty]\) a weight. Define
\[
\Phi_b(y)=\Phi(b^*yb)\qquad(b\in M,\ y\in M_+).
\]
Then $\Phi$  is normal (resp.\ $\sigma$-normal) if and only if  \(\Phi_b\) is completely additive on projections
for every \(b\in M\).   
\end{theorem} 

Thus  Haagerups's problem 1.11 (and its sequential variant) is equivalent to checking if (H) (resp.\ (H$_\sigma$)) implies
that such conjugated weights $\Phi_b$ have the appropriate additivity on projections.

\section{Conclusions} Our objective is to develop a theory of noncommutative measures which is closer in spirit to 
classical measure theory in the sense of being based on countable additivity rather 
than normality. In the setting of states countable additivity on projections was 
shown to be equivalent to each of sequential normality, and sequential weak* continuity. 
These continuity properties are consequently commutatively determined. When combined with the 
Gleason extension theorem, the result gives a corresponding description of normalized countably
additive projection measures whenever there is no type $I_2$ direct summand.

For weights, the present paper develops direct normality criteria,
various  constructions, and GNS closability methods, together
with sequential representation results under certain variants of semifiniteness.
The Appendix
proves that both normality
and sequential normality of arbitrary weights are commutatively
determined:
\[
 (H)\ \Longleftrightarrow\ \text{normality},
 \qquad
 (H_\sigma)\ \Longleftrightarrow\ \sigma\text{-normality}.
\]
(The sequential version may be moved to the present paper.) 
 The
normality proof uses the GNS and normal-minorant methods developed
here. The sequential proof also uses our quotient construction and
countably generated reduction.

In summary we have for states and a significant class of weights succeeded in exposing rich 
connections of the proposed sequentially based quantum measure theory to both $\sigma$-normality 
and sequential weak* convergence. Although much work remains to be done, it is clear 
that this theory is worthy of further investigation. 

Aspects that deserve further investigation at the very outset include the challenge of 
developing sequential versions of other key von Neumann algebraic results involving 
normality.  We remark that some of the results here have variants for monotone ($\sigma$-)complete 
$C^*$-algebras that we hope to investigate further with 
students.  Another natural next step would be to try and extend this ``sequentially based'' theory 
to noncommutative $L^p$ spaces and to explore the application of these ideas 
to concrete problems in for example noncommutative ergodic theory. In this regard we note that 
the very large existing noncommutative ergodic theory literature is usually or often based on sequential convergence.

\section{Appendix by first author on  the solution to Haagerup's weight problem}

This Appendix is written to be completely selfcontained, since it may be moved later to another paper.

Weights provide an operator-algebraic counterpart of integration
against measures of possibly infinite mass. In the theory of von
Neumann algebras, their normality is the monotone-continuity condition
that allows increasing approximations to pass through the integral.
Haagerup proved that normality of a weight is equivalent to weak*
lower semicontinuity on the positive cone, and also to representation
as a sum of bounded normal positive functionals
\cite[Theorem~1.8 and the introduction]{Haag}. In the same paper he asked whether normality is already
determined by restrictions to commutative von Neumann subalgebras
\cite[Problem~1.11, p.~310]{Haag}.

For bounded positive functionals the corresponding question follows
from the characterization of normality by complete additivity on
orthogonal projections. This projection criterion alone fails for
weights. Haagerup's example on $\ell^\infty(\mathbb N)$ is
\[
 \Phi((a_n))=
 \begin{cases}
   \sum_n a_n,&(a_n)\text{ has finite support},\\
   \infty,&\text{otherwise},
 \end{cases}
 \qquad (a_n)\in\ell^\infty(\mathbb N)_+.
\]
Its projection values agree with counting measure, but the
finite-coordinate truncations of $(2^{-n})$ increase to an element
of infinite $\Phi$-weight while their weights remain bounded
\cite[Remark~1.12]{Haag}. Thus additivity on projections must
not be confused with complete additivity on arbitrary positive
families. The latter is one of Haagerup's characterizations of
normality.

We use the following abbreviations:
\begin{align*}
 (H)\quad&\Phi|_{B_+}\text{ is normal for every unital abelian}
             \\[-2pt]
 &\hspace{2.5em}\text{von Neumann subalgebra }B\subseteq M;\\
 (H_\sigma)\quad&\Phi|_{B_+}\text{ is $\sigma$-normal for every such }B.
\end{align*}
Requiring the subalgebras to contain $1_M$ does not change these
conditions: one can adjoin $\mathbb C(1-e)$ to an abelian subalgebra
whose identity is a projection $e\in M$.

\begin{theorem}[Commutative determination] \label{thm:main2}
Let $\Phi$ be a weight on an arbitrary von Neumann algebra $M$.
\begin{enumerate}[label=\textup{(\roman*)}]
\item $\Phi$ is normal if and only if it satisfies $(H)$.
\item $\Phi$ is $\sigma$-normal if and only if it satisfies $(H_\sigma)$.
\end{enumerate}
Neither faithfulness nor any form of semifiniteness is assumed.
\end{theorem}

The first assertion answers Haagerup's question, and the second is
its sequential analogue (this will probably appear in the prequel paper). The distinction between the two conclusions
is essential on non-$\sigma$-finite algebras; an elementary example is
given in Section~\ref{sec:consequences}. We also prove that complete
(respectively, countable) additivity on orthogonal projections,
together with norm lower semicontinuity, characterizes normality
(respectively, sequential normality).

The starting point is the weight theory developed  in 
the earlier Sections.
This proves, in particular, that countable projection additivity,
sequential normality, and sequential weak* continuity are equivalent
for states on arbitrary von Neumann algebras. For weights it supplies
normal-minorant constructions, an orthogonal-series lemma, and
closability of the GNS map under $(H_\sigma)$. The present paper
strengthens this last conclusion from closability to closedness.
The proof uses finite-weight spectral projections to approximate an
element in the domain of the closure in the graph norm, and then
assembles orthogonally supported pieces using Lemma~6.15.
This gives norm lower semicontinuity at every
positive element, including elements of infinite weight.

The passage from norm lower semicontinuity to normality uses the
algebra $A_\Phi=\overline{\mathfrak m_\Phi}^{\|\cdot\|}$, where
$\mathfrak m_\Phi$ is the linear span of the finite-weight positive
elements. If this span is weak* dense, the one-sided noncommutative
Egoroff theorem shows that every nonzero projection contains a
nonzero subprojection in $A_\Phi$. Crucially, such a subprojection
need not have finite weight. We combine this approximation with the
bounded-minorant representation of lower semicontinuous weights from 
\cite{Combes,KV} and Lemma~6.11.
A reduction to the corner
generated by the finite-weight projections removes ordinary
semifiniteness. These arguments first give the general normality
assertion and its $\sigma$-finite sequential special case.
As mentioned above, the general sequential assertion will probably appear 
in a subsequent paper. 

Section~\ref{sec:prelim2} fixes the notation and proves the precise
semifiniteness and corner-reduction facts required below.
Section~\ref{sec:closedness} proves closedness of the GNS map.
Section~\ref{sec:projections} establishes projection approximation
in $A_\Phi$. Section~\ref{sec:normality} proves
normality. Section~\ref{sec:consequences} records the resulting criteria
and the distinction between normality and sequential normality.

\section{Preliminaries and semifiniteness}\label{sec:prelim2}

Throughout, $M$ is a complex von Neumann algebra, with predual $M_*$.
The weak* topology is $\sigma(M,M_*)$. We write Proj$(M)$ for its
projections, and $s(x)$ for the support projection of $x\in M_+$.
An orthogonal sum of projections always means its strong-operator
sum in $M$. A sum of nonnegative extended real numbers over an
arbitrary index set means the supremum of the finite subsums.
A von Neumann algebra is $\sigma$-finite if it admits a faithful
normal state, equivalently if every orthogonal family of nonzero
projections is countable. The $\sigma$-strong topology is generated
by the seminorms $x\mapsto\omega(x^*x)^{1/2}$, $\omega\in M_*^+$;
on norm-bounded sets it agrees with the strong operator topology
in a faithful normal representation.

A weight is \emph{countably additive on projections} if
\[
 \Phi\Bigl(\sum_i p_i\Bigr)=\sum_i\Phi(p_i)
\]
for every countable orthogonal family $(p_i)$; it is \emph{completely
additive on projections} if this holds for arbitrary orthogonal
families. It is \emph{faithful} if $\Phi(x)=0$ for $x\ge0$ implies
$x=0$. Norm lower semicontinuity means that
$\Phi(x)\le\liminf_n\Phi(x_n)$ whenever $x_n,x\in M_+$ and
$\|x_n-x\|\to0$; because the norm topology is metrizable, this also
characterizes lower semicontinuity using nets. Domination
$\rho\le\Phi$ always means domination at every point of $M_+$.

For a weight $\Phi$, put
\[
 \pphi=\{x\in M_+:\Phi(x)<\infty\},\qquad
 \nphi=\{a\in M:\Phi(a^*a)<\infty\},
\]
\[
 \mphi=\operatorname{span}\pphi,\qquad
 \Aphi=\overline{\mphi}^{\|\cdot\|}.
\]
Then $\nphi$ is a left ideal, $\mphi$ is a hereditary $*$-subalgebra,
$\mphi\cap M_+=\pphi$, and
$\mphi=\operatorname{span}\{b^*a:a,b\in\nphi\}$.
Its norm closure $\Aphi$ is a hereditary $C^*$-subalgebra of $M$.
The weight extends linearly to $\mphi$.
Complete the quotient of $\nphi$ by the null space of the positive
sesquilinear form $(a,b)\mapsto\Phi(b^*a)$. This gives the GNS
Hilbert space $\HH$ and the canonical linear map
\[
 \Lam:\nphi\longrightarrow\HH,
 \qquad \|\Lam(a)\|^2=\Phi(a^*a).
\]
Left multiplication induces the representation
$\pi_\Phi(x)\Lam(a)=\Lam(xa)$, $x\in M$, $a\in\nphi$.
These constructions do not require normality; see, for example,
\cite[Section~1]{KV} and  Section 6.
We say that $\Phi$ is
\emph{ordinarily semifinite} if $\mphi$ is weak* dense in $M$.

The following is well known but we include a proof for convenience: 

\begin{lemma}[Ordinary semifiniteness]\label{lem:semifinite}
For an arbitrary weight $\Phi$ on $M$, the following are equivalent:
\begin{enumerate}[label=\textup{(\roman*)}]
\item $\mphi$ is weak* dense in $M$.
\item $\displaystyle\bigvee\{e\in {\rm Proj}(M):\Phi(e)<\infty\}=1$.
\item There is an increasing net of positive contractions
$u_\alpha\in\pphi$ with $u_\alpha\uparrow1$.
\end{enumerate}
If $M$ is $\sigma$-finite, the net in \textup{(iii)} can be replaced
by an increasing sequence. No normality or faithfulness of $\Phi$
is required.
\end{lemma}
\begin{proof}
Let $e$ be the join in (ii). If $x\in\pphi$, then
\[
 \Chi_{[1/n,\infty)}(x)\le nx,
 \qquad \Phi\bigl(\Chi_{[1/n,\infty)}(x)\bigr)<\infty.
\]
Their join is $s(x)$, so $s(x)\le e$ and $x\in eMe$.
Thus $\mphi\subseteq eMe$. Since $eMe$ is weak* closed,
(i) implies $e=1$.

Assume (ii). Order $\pphi$ by the operator order. It is directed,
since $a+b$ is a finite-weight upper bound for $a,b\in\pphi$.
For $a\in\pphi$ set
\[
 u_a=a(1+a)^{-1}=1-(1+a)^{-1}.
\]
Inversion reverses order on positive invertible operators, so $(u_a)$
is an increasing net of positive contractions. Also $u_a\le a$,
hence $\Phi(u_a)<\infty$. Let $u$ be its strong limit.
For every finite-weight projection $f$ and every $t>0$,
\[
 u\ge u_{tf}=\frac{t}{1+t}f.
\]
Consequently $u\ge f$. Because $u\le1$, it acts as the identity on
the range of each such $f$. Their join is $1$, so $u=1$, proving (iii).

If (iii) holds and $x\in M_+$, then
\[
 0\le u_\alpha^{1/2}xu_\alpha^{1/2}\le\|x\|u_\alpha.
\]
These elements have finite weight and converge strongly to $x$.
They are uniformly bounded, hence also converge weak* to $x$.
Since $M_+$ linearly spans $M$, (i) follows.

Finally, if $\omega$ is a faithful normal state on $M$, choose
increasing indices $\alpha_n$ with
$\omega(u_{\alpha_n})\to1$. If $v=\sup_nu_{\alpha_n}$, then
$0\le v\le1$ and $\omega(v)=1$ by normality. Faithfulness gives
$v=1$, proving the sequential assertion.
\end{proof}

\begin{lemma}[The finite-weight corner]\label{lem:finite-corner}
Let $\Phi$ be any weight and put
\[
 e=\bigvee\{f\in {\rm Proj}(M):\Phi(f)<\infty\}.
\]
Every finite-weight positive element is supported on $e$, and the
restriction $\Phi_e=\Phi|_{(eMe)_+}$ is ordinarily semifinite.
Moreover:
\begin{enumerate}[label=\textup{(\roman*)}]
\item $(H)$ and $(H_\sigma)$ pass from $\Phi$ to $\Phi_e$.
\item If $\Phi_e$ is normal (respectively, $\sigma$-normal), then so
is $\Phi$.
\end{enumerate}
\end{lemma}
\begin{proof}
The support assertion was proved in the first paragraph of
Lemma~\ref{lem:semifinite}. The finite-weight projections in $eMe$
have join $e$, so the same lemma, applied with identity $e$, proves
ordinary semifiniteness. For (i), an abelian von Neumann subalgebra
$B\subseteq eMe$ with identity $e$ gives the unital abelian subalgebra
$B+\mathbb C(1-e)$ of $M$.

For (ii), let $0\le x_\alpha\uparrow x$. If $x\in eMe$, then
$0\le x_\alpha\le x$ implies $x_\alpha\in eMe$ for every $\alpha$,
and normality of $\Phi_e$ applies. If $x\notin eMe$, some
$x_\alpha\notin eMe$, since $eMe$ is weak* closed and
$x_\alpha\to x$ weak*. The support assertion gives
$\Phi(x_\alpha)=\infty$, so
$\Phi(x)=\sup_\alpha\Phi(x_\alpha)=\infty$.
The same argument with a sequence proves the sequential assertion. \end{proof}

\section{Closedness of the GNS map}\label{sec:closedness}

Theorem~6.16 above 
makes
$\Lam$ closable for the operator norm on its domain and the
Hilbert-space norm on $\HH$. Lemma~6.15 there states that if
$c_k\in\nphi$, the projections $q_k$ are mutually orthogonal,
$c_k=c_kq_k$, and
\begin{equation}\label{eq:gns-hsigma-summability}
 \sum_{k\ge1}2^k\bigl(\|c_k\|^2+\|\Lam(c_k)\|^2\bigr)<\infty,
\end{equation}
then $c=\sum_kc_k$ converges in operator norm, belongs to $\nphi$,
and $\Lam(c)=\sum_k\Lam(c_k)$ in $\HH$. 

\Needspace{10\baselineskip}
\begin{theorem}[Closedness from $(H_\sigma)$]
\label{thm:gns-hsigma-closed}
Let $\Phi$ be a weight on an arbitrary von Neumann algebra $M$.
If $(H_\sigma)$ holds, then the GNS map
\[
       \Lam:\nphi\subseteq M\longrightarrow\HH
\]
has a closed graph for the operator norm on $M$ and the Hilbert-space
norm on $\HH$.  
\end{theorem}

Let $T=\overline{\Lam}$, which is a well-defined operator by Theorem~6.16
Our task is to show
$\Dom T\subseteq\nphi$, since the opposite inclusion and agreement
of the two operators are automatic.

\begin{lemma} 
\label{lem:gns-hsigma-weak}
Suppose $d_j\in\nphi$ and $\|d_j\|\to0$, with
\[\sup_j\|\Lam(d_j)\|<\infty.\]
Then $\Lam(d_j)\to0$ weakly in $\HH$.
\end{lemma}
\begin{proof}
Every subsequence of the bounded sequence $(\Lam(d_j))$ has a weakly
convergent further subsequence. Let $\zeta$ be the resulting limit.
Mazur's lemma supplies convex combinations of tails whose images
converge in Hilbert-space norm to $\zeta$.  Their preimages still
converge to $0$ in operator norm.  Closability implies $\zeta=0$.
Thus every subsequence has a further subsequence converging weakly
to $0$, which proves the assertion. 
\end{proof}

\begin{lemma} 
\label{lem:gns-hsigma-cutoff}
For each $a\in\Dom T$ there are projections
$e_j\in M$ with $\Phi(e_j)<\infty$ such that
\begin{equation}\label{eq:gns-hsigma-graph-cutoff}
    \|a-ae_j\|\longrightarrow0,
    \qquad
    \|Ta-\Lam(ae_j)\|\longrightarrow0.
\end{equation}
Here $ae_j\in\nphi$.  Moreover, if $P\in\Proj(M)$ has finite weight
and $a=a(1-P)$, the projections $e_j$ may be chosen below $1-P$.
\end{lemma}
\begin{proof}
Choose $b_j\in\nphi$ with $b_j\to a$ in operator norm and
$\Lam(b_j)\to Ta$ in $\HH$. Write
\[
  \varepsilon_j=\|b_j-a\|,\quad
  \delta_j=\varepsilon_j+j^{-1},\quad
  C=\sup_j\Phi(b_j^*b_j)<\infty.
\]
Define the spectral projections
\[
       e_j=\Chi_{[\delta_j,\infty)}(b_j^*b_j).
\]
Since $\delta_j e_j\le b_j^*b_j$,
\begin{equation}\label{eq:gns-hsigma-weight-e}
       \Phi(e_j)\le C/\delta_j<\infty.
\end{equation}
Functional calculus also gives
$\|b_j(1-e_j)\|\le\delta_j^{1/2}$ and, because
$e_j$ commutes with $b_j^*b_j$,
\[
  \|\Lam(b_j(1-e_j))\|^2
    =\Phi((1-e_j)b_j^*b_j(1-e_j))
    \le \Phi(b_j^*b_j)\le C.
\]
The elements $d_j=b_j(1-e_j)$ therefore meet the assumptions of
Lemma~\ref{lem:gns-hsigma-weak}. Hence
\[
    \Lam(b_je_j)=\Lam(b_j)-\Lam(d_j)
                      \to Ta\quad\text{weakly}.
\]
On the other hand by a similar argument to the above, 
\[
 \limsup_j\|\Lam(b_je_j)\|^2
 \le\lim_j\|\Lam(b_j)\|^2=\|Ta\|^2.
\]
Weak convergence together with this norm inequality implies
$\Lam(b_je_j)\to Ta$ in Hilbert-space norm.

The finite weight of $e_j$ ensures $ae_j\in\nphi$, since
$e_ja^*ae_j\le\|a\|^2e_j$.  Moreover,
\begin{align*}
 \|\Lam((a-b_j)e_j)\|^2
  &=\Phi(e_j(a-b_j)^*(a-b_j)e_j)\\
  &\le\varepsilon_j^2\Phi(e_j)
   \le C\frac{\varepsilon_j^2}{\delta_j}
   \le C\varepsilon_j\longrightarrow0.
\end{align*}
Consequently $\Lam(ae_j)\to Ta$. Finally,
\[
 \|a(1-e_j)\|
 \le\|a-b_j\|+\|b_j(1-e_j)\|
 \le\varepsilon_j+\delta_j^{1/2}\longrightarrow0,
\]
which proves \eqref{eq:gns-hsigma-graph-cutoff}.

For the last assertion, suppose $a=a(1-P)$ and $\Phi(P)<\infty$.
Replace the graph approximants $b_j$ by $b_j(1-P)$.
Indeed, $b_jP\in\nphi$ and
\[
  \|\Lam(b_jP)\|^2
   =\Phi(Pb_j^*b_jP)
   \le\|b_jP\|^2\Phi(P)\longrightarrow0,
\]
because $b_jP\to aP=0$ in norm. Thus $b_j(1-P)$ still
converges to $a$ in the graph norm of $T$.  For these new
approximants, $b_j^*b_j$ (with $b_j=b_j(1-P)$) is supported in
$1-P$, so every spectral projection $e_j$ above a strictly
positive threshold lies below $1-P$.
\end{proof}

\begin{proof}[Proof of Theorem~\ref{thm:gns-hsigma-closed}]
Fix $a\in\Dom T$. We recursively construct mutually orthogonal
finite-weight projections $q_k$ $(k\ge1)$ and residuals
\[
 P_k=\sum_{j=1}^kq_j,\qquad r_k=a(1-P_k), \quad k\ge1,
\]
with the initial conventions
\[
                  P_0=0,\qquad r_0=a,
\]
such that
\begin{equation}\label{eq:gns-hsigma-decay}
    r_k\in\Dom T,\qquad
    \|r_k\|+\|Tr_k\|\le 4^{-k}, \qquad k\ge1.
\end{equation}

For $k=1$, apply Lemma~\ref{lem:gns-hsigma-cutoff} to $r_0=a$
with $P=P_0=0$. Choose a projection $q_1$ of finite $\Phi$-weight
such that
\[
 \|a-aq_1\|+\|Ta-\Lam(aq_1)\|\le\tfrac14.
\]
Set $P_1=q_1$, $r_1=a-aq_1$, and $c_1=aq_1$.
Since $aq_1\in\nphi\subseteq\Dom T$, we have
\[
 Tr_1=Ta-\Lam(aq_1),\qquad
 \|r_1\|+\|Tr_1\|\le4^{-1}.
\]

Now suppose $k\ge2$ and $P_{k-1}$ has already been constructed.
It has finite weight, and $r_{k-1}=r_{k-1}(1-P_{k-1})$.
Apply the localized assertion of Lemma~\ref{lem:gns-hsigma-cutoff} to
$r_{k-1}$ with $P=P_{k-1}$. Choose a finite-weight projection
$q_k\le1-P_{k-1}$ sufficiently far along the sequence that
\[
   \|r_{k-1}-r_{k-1}q_k\|
   +\|Tr_{k-1}-\Lam(r_{k-1}q_k)\|\le4^{-k}.
\]
Set $r_k=r_{k-1}-r_{k-1}q_k$ and $P_k=P_{k-1}+q_k$.
Because $q_k\le1-P_{k-1}$,
\[
 r_{k-1}q_k=a(1-P_{k-1})q_k=aq_k.
\]
Moreover,
\[
 0\le q_ka^*aq_k\le\|a\|^2q_k,
 \qquad \Phi(q_k)<\infty,
\]
so $aq_k\in\nphi$. Therefore $r_k\in\Dom T$ and
$Tr_k=Tr_{k-1}-\Lam(aq_k)$. The bound
\eqref{eq:gns-hsigma-decay} follows. Orthogonality of the $q_k$
is automatic from $q_k\le1-P_{k-1}$, and
$r_k=a(1-P_k)$ as required.

Put $c_k=aq_k=r_{k-1}-r_k$. Then $c_k\in\nphi$ and
$c_k=c_kq_k$ for every $k\ge1$. For $k\ge2$, using
\eqref{eq:gns-hsigma-decay},
\[
 \|c_k\|+\|\Lam(c_k)\|
 \le4^{-(k-1)}+4^{-k}=5\cdot4^{-k}.
\]
The $k=1$ term is finite because $c_1=aq_1\in\nphi$.
Hence
\[
 \sum_{k=1}^\infty 2^k
 \bigl(\|c_k\|^2+\|\Lam(c_k)\|^2\bigr)<\infty.
\]
Also, $a-\sum_{j=1}^kc_j=r_k\to0$ in operator norm. 
Lemma~6.15
applied to $(c_k,q_k)$, therefore yields
\[
       a=\sum_k c_k\in\nphi,
       \qquad
       \Lam(a)=\sum_k\Lam(c_k).
\]
Finally, the partial sums of the last series telescope:
\[
      \sum_{j=1}^k\Lam(c_j)=Ta-Tr_k\longrightarrow Ta.
\]
Thus $\Lam(a)=Ta$ for every $a\in\Dom T$, proving that
$T=\Lam$ and the graph of $\Lam$ is closed.
\end{proof}

\begin{corollary}[Norm lower semicontinuity everywhere]
\label{cor:gns-hsigma-lsc}
Under $(H_\sigma)$, if $x_j,x\in M_+$ and
$\|x_j-x\|\to0$, then
\[
           \Phi(x)\le\liminf_{j\to\infty}\Phi(x_j).
\]
 In particular,
$\Phi$ is norm lower semicontinuous on $M_+$.
\end{corollary}
\begin{proof}
This follows from Theorem~\ref{thm:gns-hsigma-closed} and
Lemma~6.8(iii).
\end{proof}

\noindent
This strengthens  Theorem~6.16 and Corollary~6.17
 from
closability to closedness, and from lower semicontinuity at finite-weight
points to lower semicontinuity everywhere.

\section{Projections in the finite-domain closure}\label{sec:projections}

We record the one-sided noncommutative Egoroff theorem in precisely
the form used below. The explicit proof of its faithful case also
fixes the meaning of the unbounded positive operator appearing in
the approximation argument.

\begin{lemma}[One-sided noncommutative Egoroff theorem]
\label{lem:egoroff}
Let $N$ be a von Neumann algebra, let $z_j\to0$ in the
$\sigma$-strong topology, let $s\in\Proj(N)$, and let
$\omega\in N_*^+$. For every $\varepsilon>0$ there are a subsequence
$(z_{j_k})$ and a projection $q\le s$ such that
\[
 \omega(s-q)<\varepsilon,
 \qquad \|z_{j_k}q\|\longrightarrow0.
\]
\end{lemma}
\begin{proof}
This the form of the noncommutative Egoroff theorem
recalled in \cite[proof of Theorem~4.3(ii)]{CHW}, with reference to
\cite[Theorem~4.13, p.~85]{Takesaki}. \end{proof}

\begin{lemma} 
\label{lem:projection-rich}
Let $\Phi$ be ordinarily semifinite on $M$. Every nonzero projection
$p\in M$ dominates a nonzero projection $q\in A_\Phi$.
Consequently every projection $p\in M$ has an orthogonal decomposition
\begin{equation}\label{eq:projection-partition}
 p=\sum_{i\in I}q_i,\qquad q_i\in A_\Phi\cap\Proj(M).
\end{equation}
If $M$ is $\sigma$-finite, there is an increasing sequence of
projections $p_n\in A_\Phi$ with $p_n\uparrow p$.
\end{lemma}
\begin{proof}
By Lemma~\ref{lem:semifinite}, choose finite-weight positive
contractions $u_\lambda\uparrow1$. Fix $p\ne0$. Choose a nonzero
normal positive functional on $pMp$, and let $s\le p$ be its support.
After normalization it restricts to a faithful normal state $\omega$
on the $\sigma$-finite corner $sMs$.
Choose increasing indices $\lambda_n$ so that
\[
 z_n=s-su_{\lambda_n}s\ge0,\qquad \omega(z_n)\le2^{-n}.
\]
The decreasing sequence $(z_n)$ has a positive infimum $z$.
Normality gives $\omega(z)=0$, and faithfulness gives $z=0$.
Thus $z_n\to0$ strongly, and hence $\sigma$-strongly because the
sequence is norm bounded. By Lemma~\ref{lem:egoroff} in $sMs$,
with $\varepsilon=1/2$, there are a subsequence and a nonzero
projection $q\le s$ such that
\[
 \omega(s-q)<\tfrac12,
 \qquad \|z_{n_k}q\|\longrightarrow0.
\]
The projection is nonzero because $\omega(s)=1$.
For $u_k=u_{\lambda_{n_k}}$ we have
\begin{align*}
 \|(1-u_k)q\|^2
 &=\|q(1-u_k)^2q\|\\
 &\le\|q(1-u_k)q\|
   =\|qz_{n_k}q\|
   \le\|z_{n_k}q\|\longrightarrow0.
\end{align*}
Here we used $0\le1-u_k\le1$ and $q\le s$.
Moreover,
\[
 0\le u_kqu_k\le u_k^2\le u_k,
 \qquad \Phi(u_kqu_k)\le\Phi(u_k)<\infty,
\]
and
\[
 \|q-u_kqu_k\|
 \le\|(1-u_k)q\|+\|u_kq(1-u_k)\|
 \le2\|(1-u_k)q\|\longrightarrow0.
\]
Thus $q\in A_\Phi$.

For the decomposition, call a family \emph{admissible} if its
members are mutually orthogonal nonzero projections dominated by
$p$ and belonging to $A_\Phi$. Order these families by inclusion.
The union of a chain is admissible: each projection belongs to one
family in the chain, and any two belong to a common family.
Zorn's lemma gives a maximal admissible family $(q_i)_{i\in I}$.
Its supremum $q=\bigvee_iq_i$ is taken in $M$ (it does not have to be in $A_\Phi$). 
If $q<p$, the first assertion supplies a
nonzero projection in $A_\Phi$ below $p-q$, contradicting maximality.
Hence $q=p$.

If $M$ is $\sigma$-finite, the nonzero members of the family are
countable. Their finite partial sums give the required increasing
sequence; a finite family is handled by repeating its final sum.
\end{proof}

\section{Normality and normal minorants}\label{sec:normality}

We shall use the normal-minorant lemma of Lemma~6.11
: if $\rho\in M^*_+$ is dominated by a weight
$\Phi$ which is completely additive on orthogonal projections,
there is $\eta\in M_*^+$ with $\eta\le\rho$ and
\[
 \eta(a)=\rho(a)\qquad(a\in\pphi).
\]
Because the functionals are bounded, this agreement extends by
linearity to $\mphi$ and then by norm continuity to $A_\Phi$.
We use the supremum representation of lower semicontinuous weights;
no additivity of an arbitrary pointwise supremum of functionals is
assumed.

\begin{lemma}[The normal-minorant envelope]\label{lem:envelope}
Suppose $\Phi$ is norm lower semicontinuous and completely additive
on orthogonal projections. Define
\[
 \Psi(x)=\sup\{\eta(x):\eta\in M_*^+,\ \eta\le\Phi\}, \qquad  x\in M_+.
\]
Then $\Psi\le\Phi$, $\Psi=\Phi$ on $(A_\Phi)_+$, and
\begin{equation}\label{eq:envelope-monotone}
 \Psi(x)=\sup_\alpha\Psi(x_\alpha)
 \qquad 0\le x_\alpha\uparrow x.
\end{equation}
The map $\Psi$ is monotone and positively homogeneous. Additivity
of $\Psi$ is not needed for these assertions.
\end{lemma}
\begin{proof}
The inequality, monotonicity, and homogeneity are immediate.
By normality of the functionals in the defining family,
\[
 \Psi(x)=\sup_{\eta\le\Phi,\,\eta\in M_*^+}
                 \sup_\alpha\eta(x_\alpha)
        =\sup_\alpha\Psi(x_\alpha),
\]
proving \eqref{eq:envelope-monotone}.

Put $A=A_\Phi$. The restriction $\Phi|_{A_+}$ is norm lower
semicontinuous and has norm-dense finite domain. Combes's
bounded-minorant representation \cite{Combes}, in the form
\cite[Theorem~1.6]{KV}, gives
\begin{equation}\label{eq:combes}
 \Phi(a)=\sup\{\rho(a):\rho\in A^*_+,\ \rho\le\Phi|_{A_+}\}, 
 \qquad a\in A_+ .
\end{equation}  Extend any such $\rho$ to a bounded positive
functional $\widehat\rho$ on $M$ by the positive Hahn--Banach
extension theorem. This extension is dominated by $\Phi$ on all
of $M_+$: the assertion is automatic where $\Phi$ is infinite,
and every positive element of finite $\Phi$-weight belongs to $A$,
where the assertion is the original domination.

Apply  Lemma~6.11
to $\widehat\rho$. The resulting
$\eta\in M_*^+$ satisfies $\eta\le\widehat\rho\le\Phi$ and
agrees with $\widehat\rho$ on $\pphi$, hence on $A$.
Thus $\eta(a)=\rho(a)$ for $a\in A_+$.
Taking the supremum in \eqref{eq:combes} yields
$\Phi(a)\le\Psi(a)$ on $A_+$. The reverse inequality was already
established.
\end{proof}

\begin{theorem}[Projection additivity and norm lower semicontinuity]
\label{thm:normal-criterion}
A weight on an arbitrary von Neumann algebra is normal if and only
if it is norm lower semicontinuous and completely additive on
orthogonal projections.
\end{theorem}
\begin{proof}
Normality implies norm lower semicontinuity by
\cite[Theorem~1.8]{Haag}; complete projection additivity follows
by applying normality to finite partial sums.

Conversely, first suppose that $\Phi$ is ordinarily semifinite,
and let $\Psi$ be as in Lemma~\ref{lem:envelope}.
For a projection $p$, choose an orthogonal decomposition
$p=\sum_iq_i$ from Lemma~\ref{lem:projection-rich}.
Write $p_F=\sum_{i\in F}q_i$ for a finite subset $F$ of the index
set. Since $A_\Phi$ is a linear space, $p_F\in A_\Phi$.
Complete projection additivity and \eqref{eq:envelope-monotone}
give
\begin{equation}\label{eq:projection-agreement}
 \Phi(p)=\sup_F\Phi(p_F)
        =\sup_F\Psi(p_F)=\Psi(p).
\end{equation}

Now let $x\in M_+$. If $\Psi(x)=\infty$, then $\Phi(x)=\Psi(x)$
because $\Psi\le\Phi$. If $\Psi(x)<\infty$, let
$p_m=\Chi_{[1/m,\infty)}(x)$. Since $p_m\le mx$, monotonicity,
homogeneity, and \eqref{eq:projection-agreement} give
\[
 \Phi(p_m)=\Psi(p_m)\le m\Psi(x)<\infty.
\]
Hence $xp_m$ has finite weight, because
$0\le xp_m\le\|x\|p_m$. Moreover $\|x-xp_m\|\le1/m$, so
$x\in A_\Phi$. The agreement in Lemma~\ref{lem:envelope} gives
$\Phi(x)=\Psi(x)$ in this case too. Thus $\Phi=\Psi$ on $M_+$,
and \eqref{eq:envelope-monotone} proves normality.  

For an arbitrary weight, pass to the finite-weight corner of
Lemma~\ref{lem:finite-corner}. Both norm lower semicontinuity and
complete projection additivity pass to that corner. Its restriction
is ordinarily semifinite, so the preceding argument makes it normal.
Lemma~\ref{lem:finite-corner}(ii) then gives normality on $M$.
\end{proof}

\begin{proof}[Proof of Theorem~\ref{thm:main2}(i)]
The forward implication follows by restriction. If $(H)$ holds,
then $(H_\sigma)$ holds, so Corollary~\ref{cor:gns-hsigma-lsc}
gives norm lower semicontinuity. Also $(H)$ gives complete
additivity on orthogonal projections by restricting to the unital
abelian von Neumann algebra generated by the family.
Apply Theorem~\ref{thm:normal-criterion}.
\end{proof}

\begin{corollary}[The $\sigma$-finite case]\label{cor:sigma-finite}
Let $M$ be $\sigma$-finite and $\Phi$ any weight satisfying
$(H_\sigma)$. Then $\Phi$ is normal.
\end{corollary}
\begin{proof}
Corollary~\ref{cor:gns-hsigma-lsc} gives norm lower semicontinuity.
Condition $(H_\sigma)$ gives countable projection additivity.
Every orthogonal family of nonzero projections in $M$ is countable,
so this is complete projection additivity. Apply
Theorem~\ref{thm:normal-criterion}.
\end{proof}

\section{Consequences and scope}\label{sec:consequences}

The following is the approximation form of the 
criteria we have obtained. We say  that $\Phi$ has \emph{spectral finite-spectrum recovery}
if, for every $x\in M_+$,
\begin{equation}\tag{$\mathrm{FS}_{\mathrm c}$}\label{eq:FS-c}
 \Phi(x)=\sup\{\Phi(b):0\le b\le x,\ b\in W^*(1,x),
                          \ b\text{ has finite spectrum}\}.
\end{equation} 

\begin{corollary}[Normal and sequential criteria]\label{cor:criteria}
For a weight $\Phi$ on any von Neumann algebra, the following
conditions are equivalent:
\begin{enumerate}[label=\textup{(\roman*)}]
\item $\Phi$ is normal.
\item $\Phi$ satisfies $(H)$.
\item $\Phi$ is completely additive on orthogonal projections and
satisfies \eqref{eq:FS-c}.
\item $\Phi$ is completely additive on orthogonal projections and
is norm lower semicontinuous.
\end{enumerate}
The corresponding four conditions with \emph{$\sigma$-normal},
\emph{$(H_\sigma)$}, and \emph{countably additive} in place of
\emph{normal}, \emph{$(H)$}, and \emph{completely additive} are
also equivalent. \end{corollary}
\begin{proof}
The equivalence of (i) and (ii), in both formulations, is
Theorem~\ref{thm:main2}. The equivalence of (ii) and (iii) is the
abelian criterion in Proposition~6.4.

If (ii) holds, Corollary~\ref{cor:gns-hsigma-lsc} gives norm lower
semicontinuity, and the relevant projection additivity follows by
restriction; hence (iv) holds. Conversely, take positive
finite-spectrum spectral approximants $b_n\le x$ with
$\|b_n-x\|\to0$. Norm lower semicontinuity and monotonicity give
\[
 \Phi(x)\le\liminf_n\Phi(b_n)\le\Phi(x),
\]
so (iv) implies \eqref{eq:FS-c} and hence (iii).
\end{proof}

\begin{corollary}\label{cor:faithful-normal}
If $\Phi$ satisfies $(H_\sigma)$ and either $M$ is $\sigma$-finite
or $\Phi$ is faithful, then $\Phi$ is normal.
\end{corollary}
\begin{proof}
The first case is Corollary~\ref{cor:sigma-finite}.
In the faithful case, $(H_\sigma)$ gives countable projection
additivity. For an orthogonal family $(p_i)$, an infinite sum
$\sum_i\Phi(p_i)$ gives $\Phi(\sum_ip_i)=\infty$ by monotonicity.
If the sum is finite, faithfulness implies that only countably many
$p_i$ are nonzero, and countable additivity applies. Thus $\Phi$
is completely additive on projections. Corollary~\ref{cor:gns-hsigma-lsc}
and Theorem~\ref{thm:normal-criterion} prove normality.
\end{proof}

{\bf Example.} The sequential conclusion cannot be replaced by normality in general.
Let $I$ be uncountable and define on $\ell^\infty(I)_+$
\[
 \Phi(f)=
 \begin{cases}
 0,&\{t:f(t)\ne0\}\text{ is countable},\\
 \infty,&\text{otherwise}.
 \end{cases}
\]
This is a weight. If $0\le f_n\uparrow f$, then
$\operatorname{supp}f=\bigcup_n\operatorname{supp}f_n$, which proves
$\sigma$-normality. But the coordinate projections all have weight
zero, whereas their supremum $1$ has infinite weight. Thus $\Phi$
is not normal. This example is also noted after Theorem~6.22.

{\bf Acknowledgements concerning the Appendix only:} DPB  thanks L. E. Labuschagne and Brent Nelson for helpful conversations. 
ChatGPT Astra 6.0 generated the main proof ideas after we suggested the theorems and provided references and suggestions for these.  This was done in an iterative process, starting with the $\sigma$-finite case.    It   was also used for mathematical discussion, proof analysis,
literature searches, exposition, and LaTeX preparation during the
development of this manuscript. The human author assume full responsibility for the correctness of all mathematical statements, and for the integrity and accuracy of all citations.

\bigskip

{\bf Acknowledgements for Sections 1--:8} The problem of whether a state on a von Neumann algebra which was countably additive on projections
is sequentially weak* continuous, was explicitly 
raised in \cite{BlecherWeaver}, and advertised by the first author at Oberwolfach in \cite[p.~1372]{B18}. 
We have worked on and off on this problem over the last 8  years or so during collaborative visits, developing many of the techniques needed, 
 and solving certain partial cases.   In 2026 we suggested an approach and certain of the above techniques  to GPT-5.6 Sol 
 to solve this and a  related and deeper problem.  It was able to do this using the strategy and some of the  techniques which we provided.  
 More details will be given in the companion
 paper \cite{BqIII}.   Then using some of these new ideas and our older techniques,  
 through an iterative process supervised and reviewed by the authors, with the first author's GPT Astra subscription
 we found in September  2026 the present
  simpler solution to the sequentially weak* continuous state problem.  After this  we analyzed the weight case through a similar
 AI-aided iterative process.   We suggested most of the theorems, and provided references and suggestions for these.
 Some of the results here had no input from AI besides cosmetic suggestions
 (we mention for example certain special cases of our main theorem, Corollary \ref{cor} and part of \ref{cor2}, Theorem \ref{thm:mainc} and Corollary \ref{cor:measure-algebra-sequential-continuity}, \ref{lem:sigma}, 
the Gleason corollary \ref{cor:gleason-sequential},  parts of Proposition \ref{prop:abelian-integral-criterion}, Theorem \ref{thm:three-strict-decomposition}, and a 
variant of Corollary \ref{cofi}).
 Indeed some of the latter results  have been known to us for many years, but we have often used shorter proofs suggested by 
 the AI. Chat-GPT  also assisted with exposition and typesetting. The human authors assume full responsibility for the correctness of all mathematical statements, and for the integrity and accuracy of all citations.   D.P. Blecher acknowledges support from NSF Grant DMS-2154903, and also from CoE-MaSS Ref No. 2025-048-OPA-Quantum Measures, and the School of Mathematical and Statistical Sciences at NWU Potchefstroom, which helped make international visits possible.  The second author acknowledges discussions with Staszek Goldstein that were helpful to him.

\end{document}